\documentclass[12pt,a4paper]{amsart}

\usepackage{amssymb}
\usepackage{amscd}
\usepackage{color}

\makeatletter
\@namedef{subjclassname@2020}{\textup{2020} Mathematics Subject Classification}
\makeatother

\def\frak{\mathfrak}
\def\Bbb{\mathbb}
\def\Cal{\mathcal}

\let\phi\varphi

\newcommand{\x}{\times}

\newcommand{\al}{\alpha}
\newcommand{\be}{\beta}

\newcommand{\la}{\lambda}

\newcommand{\ph}{\phi}
\newcommand{\ps}{\psi}

\newcommand{\si}{\sigma}
\newcommand{\ze}{\zeta}
\newcommand{\Ga}{\Gamma}
\newcommand{\La}{\Lambda}
\newcommand{\Ph}{\Phi}

\newcommand{\Om}{\Omega}

\def\Rho{\mbox{\textsf{P}}}

\newcommand{\barm}{\overline{M}}

\newcommand{\vol}{\operatorname{vol}}

\newcommand{\Ric}{\operatorname{Ric}}

\newcommand{\rpl}                         
{\mbox{$
\begin{picture}(12.7,8)(-.5,-1)
\put(0,0.2){$+$}
\put(4.2,2.8){\oval(8,8)[r]}
\end{picture}$}}

\numberwithin{equation}{section}

\newcounter{theorem}
\numberwithin{theorem}{section}
\newtheorem{thm}[theorem]{Theorem}
\newtheorem*{thm*}{Theorem}
\newtheorem{lemma}[theorem]{Lemma}
\newtheorem{prop}[theorem]{Proposition}
\newtheorem{cor}[theorem]{Corollary}
\newtheorem*{lemma*}{Lemma \thesubsection}
\newtheorem*{prop*}{Proposition \thesubsection}
\newtheorem*{cor*}{Corollary \thesubsection}

\theoremstyle{definition}
\newtheorem{definition}[theorem]{Definition}
\newtheorem*{definition*}{Definition \thesubsection}

\newtheorem*{example*}{Example \thesubsection}
\theoremstyle{remark}
\newtheorem{remark}[theorem]{Remark}
\newtheorem*{remark*}{Remark \thesubsection}

\def\sideremark#1{\ifvmode\leavevmode\fi\vadjust{\vbox to0pt{\vss
 \hbox to 0pt{\hskip\hsize\hskip1em
 \vbox{\hsize3cm\tiny\raggedright\pretolerance10000
  \noindent #1\hfill}\hss}\vbox to8pt{\vfil}\vss}}}%
                        
\begin{document}
\title{Projective Infinities and $b$-Calculus}

\author{Andreas \v Cap and A.\ Rod Gover}

\address{A.\v C.: Faculty of Mathematics\\
University of Vienna\\
Oskar--Morgenstern--Platz 1\\
1090 Wien\\
Austria\\
A.R.G.:Department of Mathematics\\
  The University of Auckland\\
  Private Bag 92019\\
  Auckland 1142\\
  New Zealand}
\email{Andreas.Cap@univie.ac.at}
\email{r.gover@auckland.ac.nz}

\begin{abstract}
  For a manifold $\barm$ with boundary $\partial M$ and interior $M$, we introduce
  and study a weakening of the concept of projective compactness for torsion-free
  linear connections on $M$, which we call projective pre-compactness. Via the
  Levi-Civita connection, this concept applies to pseudo-Riemannian metrics on
  $M$. This is motivated by applications to scattering theory and, especially, to
  general relativity (GR), in view of asymptotic forms of metrics used in these
  areas.

  In the general setting of a projectively pre-compact connection $\nabla$ we show
  that, assuming weak conditions on the asymptotics of Ricci curvature, there is an
  induced projective structure on the boundary. Under a slightly stronger condition
  on the Ricci curvature, we show explicitly that the standard tractor bundle and its
  normal tractor connection arise naturally on this boundary structure. The key
  ingredient to this is that $\nabla$ admits a smooth extension to the boundary as a
  linear connection on the tensor product of Melrose's b-tangent bundle with an
  appropriate density bundle, which then restricts to the boundary tractor bundle.

  A projectively pre-compact pseudo-Riemannian metric (satisfying the conditions on
  the asymptotics of Ricci curvature) is then shown to induce a holonomy reduction of
  the boundary projective structure to an indefinite orthogonal
  group. This  endows the boundary with a decomposition into so-called
  curved orbits, which are either open or embedded hypersurfaces, representing
  space-like, time-like and light-like infinities in a GR context. We introduce and
  study a new asymptotic form for such metrics which is available near any boundary
  point and relate it to an asymptotic form used in general relativity. The latter
  is only available near boundary points in the open curved orbits. We show that, in
  that region, projective pre-compactness essentially is equivalent to this
  asymptotic form from GR, and that projective compactness is equivalent to the
  vanishing of the mass aspect function.
  \end{abstract}

\subjclass[2020]{primary: 53C07; secondary: 83C30, 53B10, 53B21, 53C25, 53C50, 53Z05
}
\keywords{geometric compactifications, projective infinity, mass
  aspect, asymptotic forms for pseudo-Riemannian metrics}

\thanks{This work was funded in part by the Austrian Science Fund
  (FWF): 10.55776/P33559. For open access purposes, the authors have
  applied a CC BY public copyright license to any author-accepted
  manuscript version arising from this submission. A.\v C.\ thanks
  University of Auckland for hospitality. Both authors gratefully
  acknowledge support from the Royal Society of New Zealand via
  Marsden Grants 19-UOA-008 and 24-UOA-005  and from
    thematic program ``Differential Complexes: Theory, Discretization,
    and Applications'' at the Erwin Schr\"odinger Institute
    (ESI). This article is based upon work from COST Action CaLISTA
  CA21109 supported by COST (European Cooperation in Science and
  Technology). https://www.cost.eu.}

\maketitle

\pagestyle{myheadings} \markboth{\v Cap, Gover}{Projective infinities and $b$-calculus}

\section{Introduction}\label{1}
Conformal compactification was originally introduced by R.\ Penrose in the setting of
general relativity, and hence of Lorentzian metrics, with the aim to attach a
``boundary at infinity'' to non-compact space-times, and remains extremely important
\cite{Chrusciel,Fr,Penrose125}. Formally, this means finding a conformal embedding of
the space-time into a compact pseudo-Riemannian manifold and then considering the
(compact) closure of the image. A major motivation for the concept is that in
pseudo-Riemannian geometry null geodesics are conformally invariant up to
parametrization, so one can study ``where light rays meet infinity'' in this
picture. The basic concept makes sense in any signature and in the nicest cases this
leads to a conformal embedding of the space as the interior of a pseudo-Riemannian
manifold with boundary that has a specific behavior of the metric towards
infinity. In particular, this happens in many cases for complete pseudo-Riemannian
manifolds which are asymptotically hyperbolic, in an appropriate sense, and the
metric then induces a conformal structure on the boundary (i.e., the ``conformal
infinity''). Requiring the interior to be (Riemannian and negative) Einstein, one
arrives at the concept of Poincar\'e--Einstein manifolds which have become a major
research topic in geometric analysis over the last decades
\cite{AdSCFTreview,FGbook,GrZ,GuillamouAB,Ma-hodge,LeBrunH}.

Already the example of Minkowski space shows that things are
complicated in an asymptotically flat setting. Here the ``boundary''
is a union of hypersurface components which represent light-like
infinity and points glued to these, which represent space-like and
time-like infinities. While the behavior around light-like infinities
of a conformal compactification in this case is very satisfactory,
dealing with space-like and time-like infinities is subtle both from
the GR point of view and mathematically. In the GR literature attempts
to find other compactifications at space-like and time-like infinity
were made early, see e.g.\ \cite{Ashtekar-Romano} and in particular
\cite{Friedrich}, where it is indicated that such an alternative
compactification can be ``almost'' compatible with the projective
structure underlying the pseudo-Riemannian metric. In these
approaches, one obtains a boundary at infinity of hypersurface type
and there is the natural question of whether this inherits a
projective structure, making it into a \textit{projective infinity}.

The projective structure induced by a pseudo-Riemannian metric can be
viewed as the family of all \textit{geodesic paths} (or unparametrized
geodesics) of the Levi-Civita connection. More formally, one can view
it as the equivalence class of all torsion-free linear connections on
the tangent bundle, which have the same geodesics as the Levi-Civita
connection up to parametrization. In the GR literature, these
geometric aspects remain in the background, the focus is on a
certain asymptotic form for the pseudo-Riemannian metrics, in such a
compactification, near space-like and time-like infinity
\cite{Ashtekar-Romano,AM-VK}. A special case of this asymptotic form
is studied as a ``scattering metric'' in geometric scattering theory,
see Chapter 6 of \cite{Melrose}. 

One possibility to obtain a projective infinity comes from a general
  theory of projective compactifications, which was developed in
  \cite{Proj-comp,Proj-comp2} in the setting of torsion-free linear connections on
  the interior of a manifold with boundary. A change from a such a linear connection
to a projectively related one is described by a one-form. The basic definition of
projective compactness is that using a local defining function for the boundary to
construct a one-form (that does not admit a smooth extension to the boundary), the
resulting projective modification of the given connection \textit{does} admit such an
extension. So similarly to conformal compactness, where one requires the metric to
diverge in a specific way towards the boundary, one here requires the connection to
diverge in a specific way. In particular, since projective modifications of the given
connection admit a smooth extension to the boundary, the projective structure
determined by the connection admits a smooth extension.

A new feature of projective compactness is that the concept allows for and requires
an additional parameter $\beta$, called the \textit{order of projective compactness},
which turns out to be related to volume growth. This is a positive real number,
assumed to be $\leq 2$ to ensure that the boundary is at infinity, see
\cite{Proj-comp} for details. Depending on $\be$, a projective compactification may
induce different structures on the boundary, the main cases interest (so far) are
$\be=1$ and $\be=2$. For $\be=2$, one obtains a conformal structure on the boundary
and the resulting theory is rather closely related to (but different from) conformal
compactifications. The relevant case for the current article is $\be=1$. Assuming
that the boundary is totally geodesic, which amounts to asymptotic conditions on the
Ricci curvature, one here obtains an induced projective structure on the boundary.

In particular, one may apply this to the flat connection on $\Bbb
R^{n+1}$, for which, via central projection, the compactification can
be viewed as a closed hemisphere in $S^{n+1}$ with its standard (flat)
projective structure and the equator as the boundary at
infinity. Adding the Minkowski metric into the picture gives rise to
an additional structure on the boundary, which can be formalized as a
holonomy reduction in the sense of \cite{hol-red} of the canonical
Cartan connection, associated to the projective structure on the
boundary at infinity, to the group $SO(n,1)\subset SL(n+1,\Bbb
R)$. This in particular decomposes the boundary sphere into
 five so-called \textit{curved orbits}. Two of these are
  isomorphic copies of $n$-dimensional hyperbolic space, one is a copy
  of $n$-dimensional de-Sitter space, and a further two are copies of
  a conformal Riemannian sphere $S^{n-1}$.

One of the contributions of \cite{Proj-comp} is extending this picture
to general Ricci-flat connections and pseudo-Riemannian metrics, which
are projectively compact of order $1$. In the case of metrics, the work
\cite{Proj-comp} also studies the relation to asymptotic forms
$$
g=\frac{C}{\rho^4}d\rho^2+\frac{h}{\rho^2}
$$
for a local defining function $\rho$ for the boundary.  Here $C$ is
a smooth function that is asymptotic to $\pm 1$ to second order along
the boundary, i.e.~such that $C=\pm 1+\rho^2f$ for a function $f$ that
admits a smooth extension to the boundary. On the one hand, it is
shown in \cite{Proj-comp} that any metric that admits such a form is
projectively compact of order $1$. For Ricci-flat metrics,
\cite{Proj-comp} also proves a converse to this result, namely that
such an asymptotic form is always available (for appropriate defining
functions) locally around points of the open curved orbits from
above. All the developments in \cite{Proj-comp} heavily rely on
tractor calculus for projective structures, which provides an
equivalent encoding of the canonical Cartan connection and emphasises
the analogy to conformal geometry.

While these are very satisfactory results, there is a significant
drawback from the point of view of general relativity. In the
asymptotic forms considered in the physics literature, see
e.g.\ \cite{Ashtekar-Romano} and \cite{AM-VK}, one requires that the
function $C$ mentioned above has the form $\pm 1+\rho f$ and the
boundary value of $f$ (whose vanishing is equivalent to the form
required for projective compactness), called the \textit{mass aspect},
is closely related to the total mass of the space-time. Thus true
projective compactness seems to be incompatible with non-zero
mass. This was also emphasized in the recent articles
\cite{Bo-He,Bo-He2,BCH}, which explicitly use the concept of
projective compactness from \cite{Proj-comp}. However, it was noted in
these works that for the more general asymptotic form one obtains most
of the features of projective compactness.

The current article addresses and clarifies exactly this issue. We
introduce a weakening of the concept of projective compactness that we
call \textit{projective pre-compactness}. We consider the same
projective modification as in the definition of projective
compactness, but do not require that the resulting connection extends
smoothly to the boundary. Rather we define and require a very slight
weakening of this condition. Of course, this means that the projective
structure determined by a projectively pre-compact connection does
\textit{not} (in general) admit a smooth extension to the
boundary. (This is required. Indeed, as shown in \cite{Proj-comp} a
Ricci flat connection is necessarily projectively compact of order one
if its projective structure extends.)  Hence the basic tools used for
the study of projective compactness so far are not available in our
setting. Nevertheless we are able to prove that the main features of
projectively compact connections and metrics are still available in
 this  more general setting. Even in the projectively
compact case, the results we obtain are stronger than the existing
ones while proofs are significantly simpler in general.

The main technical tool we use to replace tractors is the \textit{b-tangent bundle}
introduced by R.\ Melrose and used in his school of geometric analysis (see e.g.\
\cite{Melrose}). Given a manifold $\barm$ with boundary $\partial M$ and interior
$M$, this is a vector bundle $T^b\barm\to \barm$, whose sheaf of smooth sections is
the sheaf $\frak X_{\partial}$ of vector fields on $\barm$ which are tangent to
$\partial M$ along $\partial M$. The existence of such a bundle follows from the fact
that $\frak X_{\partial}$ is a locally free sheaf. Once this vector bundle is
available, one can of course apply all the tools from differential geometry to it,
and particular the notion of a linear connection makes sense. Our key technical
result is that a projectively pre-compact linear connection on the interior $M$ of
$\barm$ extends to a linear connection on $T^b\barm(-1)$, the tensor product of
$T^b\barm$ with an appropriate density bundle. While we don't use tractors at this
stage, it should be emphasized that the idea for this approach comes from tractors,
see Section \ref{2.9} for details.

An overview of the results we obtain is as follows. Section 2 studies
projective pre-compactness in the setting of connections. After
introducing the concept, we prove that, for $\barm=M\cup\partial M$, a
projectively pre-compact linear connection $\nabla$ on $M$ gives rise
to a canonical weighted symmetric bilinear form $\Phi$ on
$T\barm|_{\partial M}\x T\partial M$, which restricts to a weighted
second fundamental form on $T\partial M\x T\partial M$, see Theorem
\ref{thm2.4}. This form is closely related to the asymptotics of the
Ricci curvature of $\nabla$, see Theorem \ref{thm2.6}. The vanishing of
this second fundamental form characterises a totally geodesic boundary
and, assuming this, we show in Theorem \ref{thm2.6} that $\nabla$
induces a projective structure on the boundary $\partial M$, so this
becomes the \textit{projective infinity} of $\nabla$.

For the next step, we need a slightly stronger assumption, namely that
the canonical weighted form $\Phi$ on $T\barm|_{\partial M}\x
T\partial M$ vanishes identically. This can also be characterized in
terms of the asymptotics of the Ricci curvature (or, equivalently, the
Schouten tensor) of $\nabla$ and we refer to is as the boundary being
\textit{strongly totally geodesic}. Under this assumption, Theorem
\ref{thm2.7} shows that $\nabla$ extends to a linear connection on the
bundle $T^b\barm(-1)$ which, in spite of the simple proof, is the main
technical input for all further developments. Our main result in the
general setting of connections is Theorem \ref{thm2.8}, which shows
that the standard tractor bundle and tractor connection associated to
the induced projective structure on the boundary arise  as
restrictions from the weighted b-tangent bundle and its induced
connection.
\begin{thm*}
  For $\barm=M\cup\partial M$, let $\nabla$ be an affine connection on
  $M$, which is projectively pre-compact and let $\Rho$ be the
  Schouten tensor of $\nabla$. Suppose further that for all
  $\xi,\eta\in\frak X(\barm)$, the function $\Rho(\xi,\eta)$ admits a
  smooth extension to the boundary if at least one of the fields lies
  in $\frak X_{\partial}(\barm)$ and its boundary value vanishes if
  both fields lie in $\frak X_{\partial}(\barm)$.  Then the bundle
  $T^b\barm(-1)|_{\partial M}$ with the connection induced by $\nabla$
  can be naturally identified with the projective standard tractor
  bundle of the induced projective structure on $\partial M$ and its
  normal tractor connection.
\end{thm*}

In Section 3, we study projective pre-compactness of Levi-Civita connections of
pseudo-Riemannian metrics. We start by considering metrics with an asymptotic form
\begin{equation*}
  g=\frac{h}{\rho^2}+\frac{C}{\rho^4}d\rho^2, \tag{3.2}
\end{equation*}
where $C=\pm 1+\rho f$ for $f\in C^\infty(\barm,\Bbb R)$ and $h\in\Ga(S^2T^*\barm)$
is non-degenerate along the boundary on $T\partial M\x T\partial M$. Following ideas
in \cite{Proj-comp}, we prove in Theorem \ref{thm3.1} that $g$ (i.e.\ its Levi-Civita
connection) is projectively pre-compact and it is projectively compact if
$f|_{\partial M}\equiv 0$.

Conversely, we next assume that $g$ is a projectively pre-compact
pseudo-Rieman\-nian metric, which is asymptotically tangentially Ricci
flat in the sense that $\Ric(\xi,\eta)$ admits a smooth extension to
the boundary if one of its entries lies in $\frak X_{\partial}(\barm)$
and vanishes along the boundary if both of its entries have this
property. Under this assumption, we show that there
is an induced projective structure on the boundary, and we  prove
in Theorem \ref{thm3.2} that, similarly as in the projectively compact
case (but even in that case under weaker assumptions), $g$ gives rise
to a reduction of projective holonomy of the boundary structure to an
 indefinite  orthogonal group. By general results on
such reductions, this decomposes $\partial M=\partial M_+\sqcup
\partial M_0\sqcup \partial M_-$ into so-called curved orbits, where
$\partial M_{\pm}$ are open and $\partial M_0$ (if non-empty) is an
embedded closed hypersurface. This is the decomposition of the
boundary into time-like, light-like and space-like infinity. It is
easy to see that an asymptotic form \eqref{g-asymp} can only be
available locally around points in $\partial M_{\pm}$ and for special
defining functions. Our second main result in Theorem \ref{thm3.3}
provides an asymptotic form that is available everywhere together with
strong additional information on the coefficients in the asymptotic
form. In this result, $\nabla^\rho$ denotes the linear connection on
$T\partial M$ induced by a specific projective modification of
$\nabla$ determined by the defining function $\rho$ and $(\la\odot
d\rho)(\xi,\eta):=\frac12(\la(\xi)d\rho(\eta)+\la(\eta)d\rho(\xi) )$.

\begin{thm*}
  For $\barm=M\cup\partial M$, let $g$ be a pseudo-Riemannian metric of
  signature $(p,q)$ on $M$ which is
projectively pre-compact and asymptotically tangentially Ricci flat, and let $\rho$ be
any defining function for $\partial M$. 

(1) There are a smooth function $C:\barm\to\Bbb R$, a one-form $\la\in\Om^1(\barm)$
and a section $h\in\Ga(S^2T^*\barm)$ such that, over $M$, we have
\begin{equation*}
g=\frac{Cd\rho^2}{\rho^4}+\frac{\la\odot d\rho}{\rho^3}+\frac{h}{\rho^2}.\tag{3.6} 
\end{equation*}
(2) Given any asymptotic form as in \eqref{g-asymp2}, let $\underline{C}$ be the
boundary value of $C$, and let $\underline{\la}\in\Om^1(\partial M)$ and
$\underline{h}\in\Ga(S^2T^*\partial M)$ the restrictions of the boundary values of
$\la$ and $h$ to tangential entries. Then $\underline{C}$ is a defining function for
$\partial M_0$ and the symmetric bilinear form on the bundle $T\partial M\oplus\Bbb R$, over $\partial M$,
defined by
\begin{equation*}
((\xi,a),(\eta,b)) \mapsto
  \underline{C}ab-\tfrac12a\underline{\la}(\eta)-
            \tfrac12 b\underline{\la}(\xi)+\underline{h}(\xi,\eta)  
\end{equation*}
is non degenerate of signature $(p,q)$. Moreover, denoting by $\nabla^{\rho}$ the
induced connection on $T\partial M$ and by $\Rho^\rho$ its Schouten tensor, we get
for $\xi,\eta,\zeta\in\frak X(\partial M)$
\begin{gather*}
  \underline{\la}= - dC \qquad
\underline{h}(\xi,\eta)=-\tfrac12(\nabla^\rho_\xi\underline{\la})(\eta)+\underline{C}\Rho^\rho(\xi,\eta)\\
(\nabla^\rho_{\xi}\underline{h})(\eta,\zeta)=\tfrac12\big(
\Rho^\rho(\xi,\eta)\underline{\la}(\zeta)+
\Rho^\rho(\xi,\zeta)\underline{\la}(\eta)\big). 
\end{gather*}
\end{thm*}
This easily leads to a proof of existence of an asymptotic form \eqref{g-asymp}
locally around points at which $C\neq 0$, see Corollary \ref{cor3.3}.

We conclude the article by proving characterizations of projective compactness in
terms of the asymptotic forms \eqref{g-asymp} and \eqref{g-asymp2}. For
\eqref{g-asymp2}, Theorem \ref{thm3.4} provides an equivalent condition for
projective compactness of $g$ in terms of asymptotic differential equations on the
coefficients $C$ and $\la$ in \eqref{g-asymp2}. In the case of \eqref{g-asymp}, it
remains to show that for $C=\pm 1+\rho f$, projective compactness of $g$ implies that
$f|_{\partial M}\equiv 0$. This is done in Theorem \ref{thm3.4.2} under the slightly
stronger curvature assumption that the  Ricci curvature of $g$ admits a smooth
extension to the boundary and its boundary value vanishes if one of its entries is
tangent to the boundary. The proof shows that this curvature assumption, which seems
to be always satisfied in GR applications, is necessary.

\section{Projective pre-compactness for connections}\label{2}
In this Section, we define and study the concept of projective pre-compactness in the
general setting of torsion-free connections. The main focus is on the existence of an
induced projective structure on the boundary and an explicit description of the
projective standard tractor bundle determined by this structure. 

\subsection{Setup and background}\label{2.1}
We consider a smooth manifold $\barm$ with boundary $\partial M$ and interior $M$. A
local defining function for $\partial M$ is then a smooth, real valued function
$\rho$ defined on an open subset $U\subset\barm$ such that
$\rho^{-1}(\{0\})=U\cap\partial M$ and such that $d\rho|_{U\cap\partial M}$ is
nowhere vanishing. A key feature of this concept is that a smooth function
$f:U\to\Bbb R$ vanishes on $U\cap\partial M$ if and only if it can be written as
$f=\rho f_1$ for some smooth function $f_1:U\to\Bbb R$. In this case we say that
$f=\Cal O(\rho)$. Similarly, for $k\geq 2$, we say that $f$ vanishes to $k$th order
along $\partial M$ or that $f=\Cal O(\rho^k)$ if and only if $f=\rho^k f_1$ for a
smooth function $f_1:U\to\Bbb R$.

For a function $f:U\cap M\to\Bbb R$, we will say that $f=\Cal
O(\rho^{-1})$, if the smooth function $\rho f:U\cap M\to\Bbb R$ admits
a (necessarily unique) smooth extension to all of $U$. Similarly for
$k\geq 2$, we say that $f=\Cal O(\rho^{-k})$ if $\rho^k f$ admits a
smooth extension. We will sometimes paraphrase this as ``$\rho^k f$ is
smooth up to the boundary''.

If $\rho:U\to\Bbb R$ is a defining function than any other defining
function on $U$ can be written as $\hat{\rho}=e^\al \rho$ for some
smooth function $\al:U\to\Bbb R$. This readily implies that all the
concepts introduced so far are independent of the choice of defining
function. Therefore, we will often not specify the domain of
definition of a defining function and say that $f=\Cal O(\rho^\ell)$
for $\ell\neq 0$ without specifying $\rho$. We will also extend this
terminology to other geometric objects like tensor fields and
differential forms.

If $f=\Cal O(\rho^{-k})$, then $\rho^k f$ has a well defined boundary
value, which is a smooth function $U\cap\partial M\to\Bbb R$. This
boundary value \textit{does} depend on the choice of defining
function, moving from $\rho$ to $\hat{\rho}=e^\al\rho$ it gets
rescaled by $e^{k\al}$. So here one has to be careful about the
dependence on the defining function.

\medskip

\begin{definition}
  The sheaf of (local) vector fields on $M$ will be denoted by $\frak
X$. This has a canonical sub-sheaf, that we denote by $\frak
X_{\partial}$, with $\frak X_{\partial}(U)$ consisting of  those
$\xi\in\frak X(U)$ which are tangent to $\partial M$ along
$U\cap\partial M$. So explicitly, we must have $\xi(x)\in T_x\partial
M\subset T_xM$ for any $x\in U\cap\partial M$. For a defining function
$\rho:U\to\Bbb R$, this may be equivalently characterized as
$d\rho(\xi)=\Cal O(\rho)$. Observe that $\frak X_{\partial}(U)$
 is a Lie subalgebra of $\frak X(U)$. Indeed for
$\xi,\eta\in\frak X_{\partial}(U)$, the fact that $0=dd\rho(\xi,\eta)$
immediately implies that $d\rho([\xi,\eta])=\xi\cdot
d\rho(\eta)-\eta\cdot d\rho(\xi)$. By assumption $d\rho(\eta)=\rho f$
for a smooth function $f:U\to\Bbb R$, so $\xi\cdot
d\rho(\eta)=d\rho(\xi) f+\rho df(\xi)=\Cal O(\rho)$ and likewise for
the second summand.
\end{definition}

\subsection{The definition of projective pre-compactness}\label{2.2}
Let $N$ be a smooth manifold and let $\nabla$ and $\hat\nabla$ be linear connections
on the tangent bundle $TN$. Recall that $\nabla$ and $\hat\nabla$ are said to be
\textit{projectively equivalent} if they have the same geodesics up to
parametrization. Since one can remove the torsion of a linear connection on $TN$
without changing the geodesics, one usually works in the setting of torsion free
connections. Then, the condition can be equivalently phrased as existence of a
one-form $\Upsilon\in\Om^1(N)$ such that
\begin{equation}\label{proj-equiv}
\hat\nabla_\xi \eta=\nabla_\xi \eta+\Upsilon(\xi)\eta+\Upsilon(\eta)\xi. 
\end{equation}
We will sometimes write this formally as $\hat\nabla=\nabla+\Upsilon$.

The concept of projective compactness was introduced in \cite{Proj-comp} as an analog
of conformal compactness of (pseudo-)Riemannian metrics. In the setting
$\barm=M\cup\partial M$ from Section \ref{2.1}, conformal compactness is a condition
on metrics $g$ on $M$ requiring that for one or equivalently any local defining
function $\rho$ for $\partial M$ the metric $\rho^2 g$ admits a smooth extension to
the boundary with non-degenerate boundary values.

Projective compactness of a linear connection $\nabla$ on $TM$ is
defined in a similar spirit by requiring that for certain one-forms
$\Upsilon$ (that are not smooth up to the
boundary ) the projective modifications
$\nabla+\Upsilon$ admit a smooth extension to the boundary. The
appropriate one-forms turn to be $\frac{d\rho}{\be\rho}$ for local
defining functions $\rho$ for $\partial M$ and a constant $\be\in
(0,2]$ which is called the \textit{order} of projective
  compactness. (Observe that $\be$ cannot be absorbed into a constant
  rescaling of $\rho$ since we are dealing with a logarithmic
  derivative here.) In general, the cases $\be=1$ and $\be=2$ are the
  most interesting ones. While the basic idea of our extension could
  also be applied for other values of $\be$ we restrict to the case
  $\be=1$ in this article.

As in \cite{Proj-comp}, we will mainly study the case of {\em special}
linear connections here, i.e.\ connections that preserve a volume
density. It is straightforward to see that any projective class admits
such connections, and in the following when we mention a connection
$\nabla$ it may be assumed that is is special, unless we state
otherwise.  So for a local defining function $\rho:U\to\Bbb R$ for
$\partial M$, we consider the one-form
$\frac{d\rho}{\rho}\in\Om^1(U\cap M)$ and for a special linear
connection $\nabla$ on $TM$, we consider the projectively modified
linear connection $\nabla^\rho:=\nabla+\frac{d\rho}{\rho}$ on $U\cap
M$. Observe that for another defining function $\hat\rho=e^\al \rho$
we get $d\hat\rho=\hat\rho d\al+e^\al d\rho$ and hence
$\frac{d\hat\rho}{\hat\rho}=\frac{d\rho}{\rho}+d\al$. Hence the
connections $\nabla^\rho$ and $\nabla^{\hat\rho}$ are projectively
related by the one-form $d\al$ that is smooth up to the boundary. This
shows that if the conditions imposed in the following definition hold
for one defining function on $U$ then they hold for any defining
function. So, loosely speaking, they are independent of the choice of
defining function.

\begin{definition}\label{def2.2}
 In our usual setting $\barm=M\cup\partial M$, consider a special
 linear connection $\nabla$ on $TM$ and the associated modified
 connection $\nabla^\rho=\nabla+\frac{d\rho}{\rho}$ determined by any
 local defining function $\rho:U\to\Bbb R$.  \\ (1) We say that
 $\nabla$ is \textit{projectively pre-compact} (of order $1$) if and
 only if the following conditions are satisfied for each
   $U$ and $\rho$:
  \begin{itemize}
  \item[(i)] $d\rho(\nabla^\rho_\xi \eta)$ is smooth up to the
    boundary for all $\xi,\eta\in\frak X(U)$.
  \item[(ii)] $\rho\nabla^\rho_\xi\eta$ is smooth up to the boundary
    for all $\xi,\eta\in\frak X(U)$.
  \item[(iii)] $\nabla^\rho_\xi\eta$ is smooth up to the boundary if
    at least one of the vector fields lies in $\frak
    X_{\partial}(U)\subset\frak X(U)$.
  \end{itemize}

\noindent (2) We say that $\nabla$ is \textit{projectively compact}
(of order $1$) if $\nabla^\rho_\xi\eta$ is smooth up to the boundary
for all $\xi,\eta\in\frak X(U)$.

\smallskip

 A pseudo-Riemannian metric $g$ on $M$ is called \textit{projectively
   pre-compact}, or, respectively, \textit{projectively compact}, (of
 order $1$) if and only if its Levi-Civita connection is.
\end{definition}

Thus projective pre-compactness should be interpreted as the fact that
$\nabla^\rho$ ``almost'' admits a smooth extension to the
boundary. Observe in particular, that we can restrict any linear
connection to operations $\frak X\x\frak X_{\partial}\to\frak X$ and
$\frak X_{\partial}\x \frak X\to\frak X$ and projective modification
makes sense in this setting. Condition (iii) then exactly says that
the projective modifications of these ``restricted connections'' are
smooth up to the boundary.

\subsection{The canonical defining density}\label{2.3}
A major tool for the study of projectively compact linear connections
is projective tractor calculus. In the setting of projective
compactness, this is automatically available, since for a projectively
compact linear connection $\nabla$, the projective structure on $M$
defined by $\nabla$ by definition admits a smooth extension to all of
$\barm$. Hence all tools from projective differential geometry are
automatically available on all of $\barm$ and several results on
smooth extendability follows straight from projective invariance.

Things are not as straightforward in the projectively pre-compact case
since one cannot expect the projective structure to extend to $\barm$.
For example, Theorem 3.3 of \cite{Proj-comp} analyzes the case of a
Ricci-flat linear connection $\nabla$ on $TM$ which preserves some
volume form on $M$. It is shown there that if the projective structure
determined by $\nabla$ extends smoothly to $\barm$, but $\nabla$
itself does not extend smoothly to a neighborhood of any boundary
point, then $\nabla$ has to be projectively compact of order $1$.

The advantage we have here (compared with say conformal geometry) is
that the bundles involved in projective tractor calculus do not need a
projective structure to be defined. This is most easily seen by
starting from the standard cotractor bundle $\Cal T^*$, which can be
defined as the first jet prolongation of a certain density bundle
$\Cal E(1)$ (that exists on any smooth manifold), see \cite{BEG} and
section 2.2 of \cite{Proj-comp} for the definitions of $\Cal E(w)$ for
$w\in\Bbb R$.

If $\nabla$ is a special linear connection, then any of the density
bundles $\Cal E(w)$ admits a non-zero section that is parallel for
(the connection induced by) $\nabla$ and hence nowhere vanishing on
$M$ and this section is unique up to a constant multiple.  The first
step in \cite{Proj-comp} is then to prove that if $\nabla$ is
projectively compact of order $1$, then any non-zero density
$\si\in\Cal E(1)$ which is parallel for $\nabla$ admits a smooth
extension by $0$ to a defining density for $\partial M$. That means
that for any connection $\hat\nabla$ in the projective class
determined by $\nabla$ and which is smooth up to the boundary,
$\hat\nabla\si$ is nowhere vanishing on $\partial M$. The jet $j^1\si$
then defines a global smooth section $L(\si)$ of $\Cal T^*=J^1\Cal
E(1)$, which is called the \textit{scale tractor} associated to
$\nabla$.

Now we can prove that these facts extend to the case where $\nabla$ is
projectively pre-compact, although the interpretation of $J^1\Cal
E(1)$ as a projective cotractor bundle is lost.

\begin{prop}\label{prop2.3}
  In our usual setting $\barm=M\cup\partial M$ let $\nabla$ be a
  special projectively pre-compact connection on $TM$
  and let $\si\in\Ga(\Cal E(1))$ be a nonzero density that is parallel
  for $\nabla$. Then $\si$ smoothly extends by $0$ to a defining
  density for $\partial M$. In particular, it determines a global
  smooth nowhere vanishing section $I:=j^1\si\in\Ga(J^1\Cal E(1))$ on
  $\barm$.
\end{prop}
\begin{proof} 
  Our strategy is to construct, locally near $\partial M$, a
  modification $\tilde\nabla$ of $\nabla$, which is projectively
  compact of order $1$ and such that $\tilde\nabla\si=0$. Then the
  claim follows from applying Proposition 2.3 of \cite{Proj-comp} to
  $\tilde\nabla$. So we fix an open subset $U$ that meets the boundary
  and a defining function $\rho:U\to \mathbb{R}$ for $\partial M$ and
  consider the projective modification
  $\nabla^{\rho}=\nabla+\frac{d\rho}{\rho}$ of $\nabla$. Choosing
  arbitrary vector fields $\xi,\eta\in\frak X(U)$, we know by
  definition that $\rho\nabla^{\rho}_\xi\eta$ admits a smooth
  extension of the boundary. Moreover, since
  $d\rho(\nabla^{\rho}_\xi\eta)$ itself is smooth up to the boundary,
  the boundary value of $\rho\nabla^{\rho}_\xi\eta$ is a smooth
  section of $T\partial M\subset T\barm |_{\partial M}$. Now the
  expression $\rho\nabla^{\rho}_\xi\eta$ evidently is linear over
  smooth functions in $\xi$ and one immediately verifies that it is
  linear over smooth functions in $\eta$ up to $\Cal
  O(\rho)$. Moreover, by definition $\rho\nabla^{\rho}_\xi\eta$ itself
  is $\Cal O(\rho)$ if either $\xi$ or $\eta$ lies in $\frak
  X_{\partial}(U)\subset\frak X(U)$ and hence in particular if one of
  them is $\Cal O(\rho)$.

  Together, these facts imply that
  $(\xi,\eta)\mapsto\rho\nabla^{\rho}_\xi\eta |_{\partial M}$ is
  induced by applying a bundle map $T\barm |_{\partial M}\x
  T\barm|_{\partial M}\to T\partial M$ to the boundary values of $\xi$
  and $\eta$ and we know in addition that this bundle map vanishes if
  one of its entries lies in $T\partial M\subset T \barm|_{\partial
    M}$. Since the quotient bundle $T \barm|_{\partial M}/T\partial M$
  is trivialized by $d\rho$, we conclude that there is a smooth
  section $\psi$ of $T\partial M$, such that the boundary value of
  $\rho\nabla^{\rho}_\xi\eta$ equals $d\rho(\xi)|_{\partial M}
  d\rho(\eta)|_{\partial M}\psi$.

  Now we can extend $\psi$ to a vector field $\tilde\ps\in\frak X(U)$
  such that $d\rho(\tilde\psi)\equiv 0$. Then $A:=d\rho\otimes
  d\rho\otimes\tilde\psi$ is a smooth section of $(S^2T^*\barm \otimes
  T\barm)_0$, where the subscript indicates the trace-free part. Of
  course $\tilde\nabla_\xi\eta=\nabla_\xi\eta-\frac1{\rho}A(\xi,\eta)$
  is a well-defined linear connection on $U\cap M$, which is
  torsion-free since $A$ is symmetric. Moreover, since $A$ is
  trace-free, $\si$ is parallel for the connection induced by
  $\tilde\nabla$. Passing to the projective modifications of the two
  connections determined by $\rho$, we of course get
  $\tilde\nabla^{\rho}_\xi\eta=\nabla^\rho_\xi\eta-\frac1{\rho}A(\xi,\eta)$. By
  construction $\rho\tilde\nabla^{\rho}_\xi\eta$ admits a smooth
  extension to the boundary for any $\xi,\eta\in\frak X(\barm)$ and
  again by construction this has vanishing boundary value. Thus
  $\tilde\nabla^{\rho}_\xi\eta$ admits a smooth extension to the
  boundary, which by definition implies that $\tilde\nabla$ is
  projectively compact of order one.
\end{proof}

Having the canonical defining density $\si$ at hand, we can actually
convert the vector field $\psi$ obtained in the proof into a canonical
weighted object, which expresses the obstruction to $\nabla$ being
projectively compact. Knowing that $\si$ is a defining density for
$\partial M$, we conclude that for a linear connection $\hat\nabla$ on
$\Cal E(1)$, the derivative $\hat\nabla\si$ is independent of the
choice of $\hat\nabla$ along the zero locus $\partial M$ of
$\si$. Since $\si$ is a defining density, we also conclude that for
any defining function $\rho$ for the boundary, $\frac{\si}{\rho}$
admits a smooth extension to the boundary with nowhere vanishing
boundary value. This readily implies that $\si\nabla^\rho_\xi\eta$
admits a smooth extension to the boundary for any $\xi,\eta\in\frak
X(\barm)$. Then the considerations in the proof above can be rephrased
as the fact that there exists a section $\mu$ of $T\partial
M(-1):=T\partial M\otimes\Cal E(-1)$ such that the boundary value of
$\si\nabla^\rho_\xi\eta$ equals $(\hat\nabla_\xi\si)|_{\partial
  M}(\hat\nabla_\eta\si)|_{\partial M}\mu$ for any $\xi,\eta$. The
only dependence on
a choice here is via the use of $\nabla^\rho$. But as observed above for any other
defining function $\hat\rho$, the difference
$\nabla^\rho_\xi\eta-\nabla^{\hat\rho}_\xi\eta$ admits a smooth extension to the
boundary, so $\mu\in\Ga(T\partial M(-1))$ is canonically associated to $\nabla$. 

\subsection{Boundary structures}\label{2.4}
Let $\nabla$ be a special linear connection on $M$ and let
$\si\in\Ga(\Cal E(1)|_M)$ be a nonzero density which is parallel for
$\nabla$. In the case that $\nabla$ is projectively compact of order
$1$, it was shown in Proposition 3.1 of \cite{Proj-comp} that
$\si\Rho_{ab}$ is smooth up to the boundary and defines a second
fundamental form along the boundary. Here $\Rho_{ab}$ is the
projective Schouten tensor which for a special connection is just a
constant multiple of the Ricci curvature. In particular, this implies
that the boundary is totally geodesic provided that $\nabla$ has the
property that its Ricci curvature admits a smooth extension to the
boundary. In this case, one then obtains an induced projective
structure on the boundary, which, in the case of Ricci-flat
projectively compact metrics, is studied in detail in
\cite{Proj-comp}.

Thus we see that additional conditions will be necessary to obtain a
boundary geometry induced by a connection $\nabla$ on $TM$ which is
only projectively pre-compact. Surprisingly, not much is lost compared
to the projectively compact case. We first discuss the connections on
densities induced by the connections $\nabla^\rho$. Recall that on a
smooth manifold $N$ of dimension $n+1$, one can define the bundle
$\vol(N)$ of volume densities as used in Section 2.2 of
\cite{Proj-comp} as the square root of the trivial line bundle
$(\La^{n+1}T^*M)\otimes(\La^{n+1}T^*M)$. The density bundle $\Cal
E(1)$ which is relevant for our purposes is then defined as the
$(n+2)^{nd}$ root of $\vol(N)$. In particular, in our setting of
$\barm=M\cup \partial M$, any linear connection on $TM$ induces a
linear connection on $\Cal E(1)|_M$, and it is natural to ask whether
such a connection is smooth up to the boundary.

\begin{lemma}\label{lem2.4}
  In our usual setting of $\barm=M\cup\partial M$, let $\nabla$ be a special linear
  connection on $TM$, which is projectively pre-compact. Let $\rho:U\to\Bbb R$ be a
  local defining function for $\partial M$ and $\nabla^\rho$ the corresponding modified
  connection. Then the induced connection $\nabla^\rho$ on $\Cal E(1)|_{U\cap M}$
  smoothly extends to all of $U$.
\end{lemma}
\begin{proof}
  Fixing a non-zero section $\si\in\Ga(\Cal E(1)|_U)$ which is parallel for $\nabla$,
  we know from Proposition \ref{prop2.3} that $\si$ smoothly extends by $0$ to a
  defining density for $U\cap\partial M$. By definition, this means that there is a
  smooth, nowhere vanishing section $\si^\rho\in\Ga(\Cal E(1)|_U)$ such that $\si$
  coincides with $\rho\si^\rho$ on $U\cap M$. But over $U\cap M$, $\nabla^\rho$ is
  obtained from $\nabla$ via a smooth projective change, so we can use the usual
  formula for the change of the induced connection on densities there. Hence for any
  vector field $\xi\in\frak X(U\cap M)$, we get
  $$
  0=\nabla_\xi(\rho\si^\rho)=\nabla^\rho_\xi
  (\rho\si^\rho)-\frac{d\rho(\xi)}{\rho}\rho\si^\rho=\rho\nabla^\rho_\xi\si^\rho.
  $$
  Since $\rho$ is nowhere vanishing on $U\cap M$, we conclude that
  $\nabla^\rho_\xi\si^\rho=0$ on $U\cap M$.

  But now any section $\tau$ of $\Cal E(1)|_U$ can be written as $f\si^\rho$ for a
  function $f\in C^\infty(U,\Bbb R)$ and thus for $\xi\in\frak X(U)$, we get
  $\nabla^\rho_\xi \tau=df(\xi)\si^\rho$ which is smooth up to the boundary.
\end{proof}

This easily brings us to a canonical weighted second fundamental form and indeed a
bit further than that. Starting from $\nabla$, we can take the smooth section
$\si\in\Ga(\Cal E(1))$ and then form $\nabla^\rho\si$ which is a well defined smooth
section of $T^*M(1)$.

\begin{thm}\label{thm2.4}
  The section $\nabla^\rho\nabla^\rho\si$ of $\otimes^2T^*M(1)$ is symmetric and
  admits a smooth extension to the boundary. Denoting the boundary value by
  $\Ph^\rho$, $\Ph^\rho(\xi,\eta)$ is independent of $\rho$ if at least one of the
  vector fields $\xi,\eta$ lies in $\Ga(T\partial M)\subset\Ga(TM|_{\partial M})$, so we denote
  it by $\Ph$ in this case. If $\xi\in\frak X(\barm)$ and $\eta\in \frak
  X_\partial(\barm)$, then $\Phi(\xi|_{\partial M},\eta|_{\partial M})$ is the
  boundary value of
  $\frac{d\rho(\xi)d\rho(\eta)}{\rho}-d\rho(\nabla^{\rho}_\xi\eta)$.  If also $\xi\in
  \frak X_\partial (M)$, then the first summand vanishes along the boundary, so
  $\Phi|_{T\partial M\x T\partial M}$ defines a canonical weighted second fundamental
  form for $\partial M$.
\end{thm}
\begin{proof}
  Take vector fields $\xi,\eta\in\frak X(\barm)$. Over the interior, we get
  $\nabla^\rho\nabla^\rho\si(\xi,\eta)=\nabla^\rho_\xi\nabla^\rho_\eta\si-
  \nabla^\rho_{\nabla^\rho_\xi\eta}\si$. By Lemma \ref{lem2.4}, the first summand
  is smooth up to the boundary. In the proof of this Lemma, we have seen that
  $\nabla^\rho\si=d\rho\otimes\si^\rho$, so the second summand can be written as
  $d\rho(\nabla^\rho_\xi\eta)\si^\rho$, which by definition is smooth up to the
  boundary, too. This observation also shows that
  \begin{equation}\label{fund-form}
  \nabla^\rho\nabla^\rho\si(\xi,\eta)=(\xi\cdot
  d\rho(\eta)-d\rho(\nabla^\rho_{\xi}\eta))\si^\rho.
\end{equation}
In the interior, this expression is visibly symmetric since $dd\rho=0$ and since
$\nabla^\rho$ is torsion-free as a projective modification of the torsion-free
connection $\nabla$.

  For a second defining function $\hat\rho=e^f\rho$, we get
  $d\hat\rho=e^fd\rho+\hat\rho df$ and inserting this, we see that 
  \begin{equation}\label{tech-diff}
    \xi\cdot
    d\hat\rho(\eta)= e^fdf(\xi)d\rho(\eta)+e^f\xi\cdot d\rho(\eta)+
    e^fd\rho(\xi)df(\eta)+\Cal O(\rho)
  \end{equation}
  On the other hand,
  $\nabla^{\hat\rho}_\xi\eta=\nabla^\rho_\xi\eta+df(\xi)\eta+df(\eta)\xi$, and hence
  $-d\hat\rho(\nabla^{\hat\rho}_{\xi}\eta)$ can be written as
\begin{equation}\label{tech-dhatrho}
  -e^f\left(d\rho(\nabla^\rho_\xi\eta+df(\xi)\eta+
    df(\eta)\xi) + df(\rho\nabla^\rho_\xi\eta)\right)+\Cal O(\rho).   
  \end{equation}
  Finally, since $\si^\rho$ is characterized by $\rho\si^\rho=\si$, we see that
  $\si^{\hat\rho}=e^{-f}\si^\rho$, so using \eqref{fund-form}, we see that
  $$
  \nabla^{\hat\rho}\nabla^{\hat\rho}\si(\xi,\eta)=\nabla^\rho\nabla^\rho\si(\xi,\eta)-
  df(\rho\nabla^\rho_\xi\eta)\si^\rho+\Cal O(\rho). 
  $$
  If either $\xi$ or $\eta$ lies in $\frak X_\partial(\barm)$, then 
  $\rho\nabla^\rho_\xi\eta=\Cal O(\rho)$, hence in this case $\rho$ and $\hat\rho$
  lead to the same boundary value. Finally, if $\eta$ lies in $\frak
  X_\partial(\barm)$, then $\xi\cdot
  d\rho(\eta)=\tfrac{d\rho(\xi)d\rho(\eta)}{\rho}+\Cal O(\rho)$ and if also $\xi$
  lies in $\frak X_\partial(\barm)$, this itself is $\Cal O(\rho)$, which completes the
  proof.
\end{proof}


\subsection{The case of totally geodesic boundary}\label{2.5}
We next study vanishing of (parts of) the canonical bilinear form $\Phi$ from
Theorem \ref{thm2.4}. 

\begin{prop}\label{prop2.5}
For $\barm=M\cup\partial M$ let $\nabla$ be a special linear connection on $TM$ which
is projectively pre-compact and has the property that the weighted second fundamental
form $\Phi|_{T\partial M\x T\partial M}$ from Theorem \ref{thm2.4}
vanishes identically. 

Then for any local defining function $\rho$, the connection
$\nabla^\rho$ restricts to give a torsion-free linear connection on
$T\partial M\to\partial M$ and any other defining function $\hat\rho$
leads to a projectively equivalent connection. Hence the boundary
$\partial M$ inherits a projective structure from $\nabla$.
\end{prop}
\begin{proof}
  By Theorem \ref{thm2.4} vanishing of $\Phi$ on $T\partial M\x
  T\partial M$ is equivalent to the fact that for $\xi,\eta\in\frak
  X_\partial(\barm)$, $d\rho(\nabla^\rho_\xi\eta)$ vanishes along the
  boundary, so $\nabla^\rho$ maps $\frak X_\partial(\barm)\x \frak
  X_{\partial}(\barm)$ to $\frak X_{\partial}(\barm)$. Given vector
  fields $\underline{\xi}$ and $\underline{\eta}$ on $\partial M$, we
  can choose smooth extensions to $\xi,\eta\in\frak X_\partial(\barm)$
  and then consider the boundary value of $\nabla^\rho_\xi\eta$ which
  lies in $\frak X(\partial M)$. One immediately verifies that this
  depends only on $\underline{\xi}$ and $\underline{\eta}$ and not on
  the choice of the extensions. Thus we obtain a well defined
  operation $\frak X(\partial M)\x\frak X(\partial M)\to\frak
  X(\partial M)$, which clearly defines a linear connection on
  $T\partial M$.

  Torsion-freeness of $\nabla^\rho$ then immediately implies torsion-freeness of the
  induced connection. Finally, we have seen already that for $\hat\rho=e^f\rho$, we
  get $\nabla^{\hat\rho}_\xi\eta=\nabla^\rho_\xi\eta+df(\xi)\eta+df(\eta)\xi$. Since
  the function $f$ is smooth up to the boundary, so is $df$ and the restriction of
  $df$ to $T\partial M$ defines an element of $\Om^1(\partial M)$, which describes a
  projective equivalence between the induced connections on the boundary.
\end{proof}

Of course, there is a natural strengthening of this condition, which will turn out to
be of crucial importance in what follows, namely that $\Phi$ vanishes identically on
all of $T\barm|_{\partial M}\x T\partial M$. By Theorem \ref{thm2.4}, this is the
case if and only if for any $\xi\in\frak X(\barm)$ and $\eta\in\frak
X_{\partial}(\barm)$, $d\rho(\nabla^\rho_\xi\eta)$ and
$\frac{d\rho(\xi)d\rho(\eta)}{\rho}$ have the same boundary value.

\begin{definition}\label{def2.5}
  (1) If $\nabla$ is projectively pre-compact and is such that the
  weighted second fundamental form $\Phi |_{T\partial M\x
      T\partial M}$ from Theorem \ref{thm2.4} vanishes, then we say
  that $\nabla$ is \textit{projectively pre-compact with totally
    geodesic boundary}.

  (2) In the situation of (1) for any of the connections $\nabla^{\rho}$, we will
  denote the induced linear connection on the boundary also by $\nabla^\rho$.

  (3) If $\nabla$ is projectively pre-compact and such that the bilinear form $\Phi$
  from Theorem \ref{thm2.4} vanishes on $TM|_{\partial M}\x T\partial M$, then we say
  that $\nabla$ is \textit{projectively pre-compact with strongly totally geodesic
    boundary}.
\end{definition}

\subsection{Ricci curvature and second fundamental form}\label{2.6}
We next analyze the curvature asymptotics of the modified connections
$\nabla^\rho$. Recall that for a smooth manifold $N$ and a general linear connection
on $TN$, there is a Ricci-type contraction of the curvature, which defines a section
of $\otimes^2T^*N$. This is symmetric if and only if the connection preserves a
volume density.

\begin{lemma}\label{lemma2.6}
  For $\barm=M\cup\partial M$ consider a special linear connection $\nabla$ on $TM$
  which is projectively pre-compact. For a local defining function $\rho$ for
  $\partial M$, consider the modified connection $\nabla^\rho$ on $TM$ and let
  $R^\rho$ be its curvature and $\Ric^\rho$ its Ricci curvature.

  (1) For $\xi,\eta,\ze\in\frak X(\barm)$ the vector field
  $\rho R^\rho(\xi,\eta)(\ze)$ admits a smooth extension to the boundary. Moreover,
  $R^\rho(\xi,\eta)(\ze)$ itself admits a smooth extension to the boundary if both
  $\xi$ and $\eta$ are in $\frak X_\partial(\barm)$.

  (2) For $\xi,\eta\in\frak X(\barm)$, the function $\rho\Ric^\rho(\xi,\eta)$ admits
  a smooth extension to the boundary. Moreover, $\Ric^\rho(\xi,\eta)$ admits a smooth
  extension to the boundary if both $\xi$ and $\eta$ lie in
  $\frak X_\partial(\barm)$.

  (3) Suppose that $\nabla$ is projectively pre-compact with totally geodesic
  boundary. Then $R^\rho(\xi,\eta)(\ze)$ admits a smooth extension to the boundary if
  $\ze$ lies in $\frak X_\partial(\barm)$ and $\Ric^\rho(\xi,\eta)$ also admits a smooth
  extension to the boundary if one of its entries lies in $\frak X_{\partial}(\barm)$.
\end{lemma}
\begin{proof}
  We directly use the defining equation
  \begin{equation}
    \label{def-R}
    R^\rho(\xi,\eta)(\ze)=\nabla^\rho_\xi\nabla^\rho_\eta\ze-\nabla^\rho_\eta\nabla^\rho_\xi\ze-
    \nabla^\rho_{[\xi,\eta]}\zeta. 
  \end{equation}
  Observe that by skew symmetry of $R^\rho$ in $\xi$ and $\eta$, we may always assume
  that one of the two fields, say $\eta$, satisfies $d\rho(\eta)\equiv 0$. (Locally fix a
  vector field $\xi_0$ such that $d\rho(\xi_0)\equiv 1$ and then expand
  $\xi=d\rho(\xi)\xi_0+(\xi-d\rho(\xi)\xi_0)$ and like-wise for $\eta$.)

  In the interior, we have $\rho R^\rho (\xi,\eta)(\zeta)=R^\rho(\xi,\eta)(\rho\ze)$ and
  since $\rho\ze$ lies in $\frak X_\partial(\barm)$, replacing $\ze$ by $\rho\ze$ in
  \eqref{def-R}, the last summand in the right hand side admits a smooth extension to
  the boundary. Likewise, $\nabla^\rho_\xi\rho\ze$ admits a smooth extension and
  since $\eta\in \frak X_\partial(\barm)$, we see that
  $\nabla^\rho_\eta\nabla^\rho_\xi\rho\ze$ admits a smooth extension to the
  boundary. Finally, by our assumption on $\eta$, we have
  $\nabla^\rho_\eta\rho\ze=\rho\nabla^\rho_\eta\ze$ and since this is in $\frak
  X_\partial(\barm)$, $\nabla^\rho_\xi\nabla^\rho_\eta\rho\ze$ also is smooth up to
  the boundary. Hence $\rho R^\rho(\xi,\eta)(\zeta)$ always admits a smooth extension
  to the boundary.

  If both $\xi$ and $\eta$ lie in $\frak X_{\partial}(\barm)$, then so does
  $[\xi,\eta]$, and hence $\nabla^\rho_\xi\ze$, $\nabla^\rho_\xi\ze$ and
  $\nabla^\rho_{[\xi,\eta]}\ze$ all admit smooth extensions. This in turn shows that
  also $\nabla^\rho_\xi\nabla^\rho_\eta\ze$ and $\nabla^\rho_\eta\nabla^\rho_\xi\ze$
  extend smoothly, and this completes the proof of (1).

  In the setting of (3), assume that $\ze$ and without loss of generality $\eta$ lie
  in $\frak X_\partial(\barm)$. Then the last two summands in \eqref{def-R} cause no
  problems while the assumptions imply that $\nabla^\rho_\eta\ze$ lies in $\frak
  X_{\partial}(\barm)$ so $\nabla^\rho_\xi\nabla^\rho_\eta\ze$ admits a smooth
  extension to the boundary, too.

  \smallskip
  
  Switching to the Ricci curvature, the fact that $\rho\Ric^\rho$ admits a smooth
  extension to the boundary follows immediately from (1) since $\Ric^\rho$ is a
  contraction of $R^\rho$. For the second part, this is not as easy, since when
  computing $\Ric^\rho(\xi,\eta)$ as a trace, $\xi$ and $\eta$ are not both inserted
  into the form slots of $R^\rho$. As in (1), we locally choose $\xi_0\in\frak
  X(\barm)$ such that $d\rho(\xi_0)\equiv 1$ and extend this to a local frame by
  $\xi_i\in\frak X(\barm)$, $i=1,\dots,n$ such that $d\rho(\xi_i)\equiv 0$ for $i\geq
  1$. Then the dual coframe has the form $d\rho$, $\al_i$ for $i=1,\dots,n$, and by
  definition we get (on $M$)
  \begin{equation}
    \label{def-Ric}
    \Ric^\rho(\xi,\eta)=d\rho(R^\rho(\xi,\xi_0)(\eta))+\textstyle\sum_{i\geq
      1}\al_i(R^\rho(\xi,\xi_i)(\eta)). 
  \end{equation}
  Since $\xi,\xi_i\in\frak X_\partial(\barm)$ each term in the sum in the right hand side
  admits a smooth extension to the boundary. For the first term in the right hand
  side, we have to go back to the defining equation \eqref{def-R}. But by
  construction $\nabla_{[\xi,\xi_0]}\eta$, $\nabla_{\xi_0}\eta$ and hence
  $\nabla_\xi\nabla_{\xi_0}\eta$, and $\nabla_\xi\eta$ admit smooth extension to the
  boundary. But this implies that also $d\rho(\nabla_{\xi_0}\nabla_\xi\eta)$ admits a
  smooth extension, so the result follows.

  Under the stronger assumption of (3), we know from above that all curvature terms
  in the right hand side of \eqref{def-Ric} admit a smooth extension to the boundary
  if $\eta$ lies in $\frak X_{\partial}(\barm)$. By symmetry of $\Ric^\rho$ in the
  interior, this implies the last part of (3).
\end{proof}

Recall that the Schouten tensor of a linear connection on the tangent bundle $TN$ of
a smooth manifold $N$ is an equivalent encoding of the Ricci curvature, which mixes
the symmetric and skew symmetric parts with different factors. In particular, for a
special connection, it is just $\frac{1}{\dim(N)-1}$ times the Ricci curvature. Thus we see that the analog of the statements on $\Ric^\rho$ from Lemma
\ref{lemma2.6} also hold for the Schouten tensor $\Rho^\rho$ of $\nabla^\rho$.

\begin{thm}\label{thm2.6}
  Consider $\barm=M\cup\partial M$ and a special connection $\nabla$ on $TM$ which is
  projectively pre-compact. Let $\si\in\Ga(\Cal E(1)|_M)$ be a nonzero section that
  is parallel for $\nabla$ and let $\Rho$ be the Schouten tensor of $\nabla$. Then
  the section $\si\Rho$ of $S^2T^*M(1)$ admits a smooth extension to the boundary and
  its boundary value restricts on $T\partial M\x T\partial M$ to the canonical
  weighted second fundamental form $\Phi$ from Theorem \ref{thm2.4}. Moreover, the form
  $\Phi$ vanishes on $T\barm|_{\partial M}\x T\partial M$ if and only if the boundary
  value of $\si\Rho$ has the same property.
\end{thm}
\begin{proof}
  On $M$, we can write $\si\Rho=\nabla\nabla\si+\si\Rho$ and it is well known that
  the right hand side defines a projectively invariant second order differential
  operator $\Ga(\Cal E(1))\to \Ga(S^2T^*M(1))$, see \cite{BEG}. Hence we may as well
  express it in terms of the projectively equivalent connection $\nabla^\rho$, whence
 \begin{equation}\label{Hessian}
\si\Rho=\nabla^\rho\nabla^\rho\si+\si\Rho^\rho. 
 \end{equation}
Now the two terms on the right hand side admit a smooth extension to the
boundary. For the first term this follows from Theorem \ref{thm2.4} and for the
second term this follows from Lemma \ref{lemma2.6}, taking into account that
$\si=\rho\si^\rho$. This proves the first claim. For the second claim, we know that
for $\xi,\eta\in\frak X_\partial(\barm)$, the function $\Rho^\rho(\xi,\eta)$ admits a
smooth extension to the boundary (again using Lemma \ref{lemma2.6}), so
$\si\Rho^\rho(\xi,\eta)$ vanishes along the boundary. Thus
$\si\Rho(\xi,\eta)=\nabla^\rho\nabla^\rho\si(\xi,\eta)$ and the boundary value of
this is the canonical weighted second fundamental form from Theorem
\ref{thm2.4}. Hence if $\si\Rho$ vanishes on $T\partial M\x T\partial M$, then we are
in the case of totally geodesic boundary, so by part (3) of Lemma \ref{lemma2.6},
$\Rho^\rho(\xi,\eta)$ admits a smooth extension the boundary if one of its entries
lies in $\frak X_\partial(\barm)$. Hence we see that, in this situation, the boundary
value of $\si\Rho$ again coincides with $\Phi$.
\end{proof}

\begin{cor}\label{cor2.6}
  For $\barm=M\cup\partial M$, let $\nabla$ be a linear connection on $M$ that is
  projectively pre-compact and let $\Rho$ be its Schouten tensor. Then the boundary
  is totally geodesic if and only if $\Rho(\xi,\eta)$ admits a smooth extension to
  the boundary whenever $\xi,\eta\in\frak X_{\partial}(\barm)$ and strongly totally
  geodesic if and only if this holds if one of the two fields lies in $\frak
  X_\partial(\barm)$.

  Moreover, if $\nabla$ is projectively pre-compact with strongly geodesic boundary,
  then for any $\xi,\eta\in\frak X(\barm)$, we get
  \begin{equation}\label{fund-form2}
   d\rho(\nabla^{\rho}_\xi\eta)=\xi\cdot
   d\rho(\eta)+\rho\Rho^\rho(\xi,\eta)-\rho\Rho(\xi,\eta).  
  \end{equation}
\end{cor}
\begin{proof}
The first part follows immediately from Theorem \ref{thm2.6}, while
\eqref{fund-form2} follows readily from combining \eqref{fund-form}
with \eqref{Hessian}.
\end{proof}

\subsection{Some b-calculus}\label{2.6a}
For the next step we  use the description of the sheaf $\frak
  X_{\partial}$ on $\barm$ developed by the Melrose school in geometric
analysis. Since $\frak X_{\partial}$ is a locally free sheaf, there is a smooth
vector bundle $T^b\barm\to\barm$, called the \textit{$b$-tangent bundle}, whose sheaf
of smooth sections coincides with $\frak X_\partial$. So by definition, a section
$\xi$ of $T^b\barm$ is a special vector field, but one has to carefully distinguish
between the two points of view, in particular if boundary values are concerned. As a
vector field, $\xi$ has a boundary value which is a vector field on $\partial M$,
i.e.~a smooth section of $T\partial M$. But $\xi$ also has a boundary value as a
section of $T^b\barm$, which is a section of $T^b\barm|_{\partial M}$, a bundle whose rank is
one larger than the one of $T\partial M$. Indeed, mapping $\xi\in\Ga(T^b\barm)$ to
its boundary value as a vector field defines a surjective bundle map
$T^b\barm|_{\partial M}\to T\partial M$. A neat way to express things is that for
$\xi\in\frak X(M)$, we have $\xi\in\Ga(T^b\barm)$ if and only if
$d\rho(\xi)|_{\partial M}=0$ for one or equivalently any local defining function
$\rho$ i.e.\ iff $d\rho(\xi)=\Cal O(\rho)$. And the boundary value of $\xi$ as a
section of $T^b\barm$ vanishes iff $\xi=\rho\xi_1$ for some $\xi_1\in\Ga(T^b\barm)$
and hence iff $d\rho(\xi)=\Cal O(\rho^2)$.

\begin{lemma}\label{lem2.6a}
  There is a short exact sequence
  \begin{equation}\label{Tb-exact}
  0\to \partial M\x \Bbb R\to T^b\barm|_{\partial M}\to T\partial M\to 0.
  \end{equation}
  A canonical trivializing section of the first bundle can be realized as the
  boundary value of $\rho\nu\in\Ga(T^b\barm)$, where $\rho$ is a local defining
  function for the boundary and $\nu\in\frak X(\barm)$ is a vector field such that
  $d\rho(\nu)|_{\partial M}\equiv 1$.
\end{lemma}
\begin{proof}
  By construction, the vector field $\rho\nu$ vanishes along $\partial M$, so
  $\rho\nu\in\Ga(T^b\barm)$. Since $d\rho(\rho\nu)=\rho+\Cal O(\rho^2)$ and this is
  not $\Cal O(\rho^2)$, we conclude that the boundary value of $\rho\nu$ as a section
  of $T^b\barm$ spans the kernel of $T^b\barm|_{\partial M}\to T\partial M$, and it
  remains to show that this boundary value is independent of the choice of $\rho$ and
  $\nu$.

Any other choice for $\nu$ is of the form $\hat\nu=\nu+\eta$, where $\eta\in\frak
X(\barm)$ is such that $d\rho(\eta)|_{\partial M}\equiv 0$. But then
$\eta\in\Ga(T^b\barm)$ and hence $\rho\hat\nu=\rho\nu+\rho\eta$ has the same boundary
value as $\rho\nu$. Any other defining function is of the form $\hat\rho=e^f\rho$ and
then $d\hat\rho|_{\partial M}=e^fd\rho|_{\partial M}$. Hence given a choice $\nu$ for
$\rho$, we can put $\hat\nu=e^{-f}\nu$ and then $\hat\rho\hat\nu=\rho\nu$, which
completes the proof.
\end{proof}

Once the vector bundle $T^b\barm$ is at hand, we can apply all the standard tools
from differential geometry to it. This immediately allows us to form the weighted
$b$-tangent bundle of weight $w\in\Bbb R$ by putting $T^b\barm(w):=T^b\barm\otimes
\Cal E(w)$. We can also apply vector bundle constructions, e.g.\ form the dual bundle
$(T^b\barm)^*$, the symmetric square $S^2T^b\barm$ or the symmetric square
$S^2(T^b\barm)^*$ of the dual. For the bundles constructed from $T^b\barm$ (without
dualization) the interpretation of sections is rather easy starting from tensor
products of sections of $T^b\barm$, which locally span the space of all sections. For
example, a section $b$ of $S^2T^b\barm$ is a section of $S^2TM$, which contracted
with $d\rho$ gives a section of $T^b\barm$ that is $\Cal O(\rho)$ and hence
contracting $b$ with two copies of $d\rho$ the result is a function that is $\Cal
O(\rho^2)$.

Things are a bit more subtle for the dual. As we have seen above, locally around a
boundary point, we can form a local frame for $T^b\barm$ consisting of
$\rho\nu$ and vector 
fields $\ze_1,\dots,\ze_n$ such that $d\rho(\ze_i)\equiv 0$. Now we can choose
$\al_i\in\Om^1(M)$ such that $\al_i(\nu)=0$ and $\al_i(\xi_j)=\delta_{ij}$ and add to
this the form $\frac{d\rho}{\rho}$ on $M$. Since for $\xi\in\Ga(T^b\barm)$,
$d\rho(\xi)=\Cal O(\rho)$, the function $\frac{d\rho(\xi)}{\rho}$ admits a smooth
extension to the boundary. Thus  $\frac{d\rho}{\rho}$ defines a smooth section
of $(T^b\barm)^*$ over all of $\barm$ and clearly the $\al_i$ also define such
sections, so together these form the dual of the frame
$\{\rho\nu,\ze_1,\dots,\ze_n\}$. This leads to a neat description of sections of
$(T^b\barm)^*$: A section of this bundle is a one-form $\phi$ on $M$ such that
$\rho\phi$ admits a smooth extension to the boundary and such that $\phi(\xi)$ admits
a smooth extension to the boundary if $\xi\in\frak X_\partial(\barm)$.

In particular, we have $\Om^1(\barm)\subset \Ga((T^b\barm)^*)$ corresponding to the
inclusion $T^*\barm\hookrightarrow (T^b\barm)^*$ dual to the projection $T^b\barm\to
T\barm$. The elements in this subspace can be easily characterized. Choose a (local)
defining function $\rho$ and a vector field $\nu$ such that $d\rho(\nu)|_{\partial
  M}\equiv 1$. Then $\rho\nu\in\Ga(T^b\barm)$ and we claim that if
$\phi\in\Ga(T^b\barm)^*$ satisfies $\phi(\rho\nu)|_{\partial M}=0$ then
$\phi\in\Om^1(\barm)$. Indeed, we know that $\rho\phi\in\Om^1(\barm)$, so it suffices
to prove that this vanishes along the boundary. But $\xi\in\frak X(M)$ can be written
as $\xi_1+d\rho(\xi)\nu$ where $\xi_1:=\xi-d\rho(\xi)\nu$ satisfies
$d\rho(\xi_1)=\Cal O(\rho)$ and hence $\xi_1\in\frak X_{\partial}(\barm)$. Thus
$\phi(\xi_1)$ admits a smooth extension to the boundary so $\rho\phi(\xi_1)$ vanishes
along the boundary. On the other hand
$\rho\phi(d\rho(\xi)\nu)=d\rho(\xi)\phi(\rho\nu)$, so this is $\Cal O(\rho)$ by
assumption.

Having these descriptions at hand, the passage to the results of tensorial
constructions is easy starting from tensor products of sections. For example, a
section $b$ of $S^2(T^b\barm)^*$ is a section of $S^2T^*M$ such that $\rho^2b$ admits
a smooth extension to the boundary and for $\xi\in\frak X_\partial(\barm)$, $\rho
b(\xi,\_)$ admits a smooth extension to the boundary as a section of $(T^b\barm)^*$.
Hence in particular for $\eta\in\frak X_{\partial}(\barm)$, $b(\xi,\eta)$ admits a
smooth extension to the boundary. As above, we have $\Ga(S^2T^*\barm)\subset
\Ga(S^2(T^b\barm)^*)$ and one immediately shows that a section $b$ lies in this subspace
if for $\rho$ and $\nu$ as above and $\xi\in\frak X_{\partial}(\barm)$, we have
$b(\rho\nu,\xi)=\Cal O(\rho)$ and $b(\rho\nu,\rho\nu)=\Cal O(\rho^2)$.

\subsection{A canonical connection}\label{2.7}
We can consider linear connections on bundles constructed from
$T^b\barm$, to which all the general results about linear connections
can be applied. In particular, starting from a projectively
pre-compact linear connection $\nabla$ on $M\subset\barm$, we have an
induced connection on any density bundle $\Cal E(w)$ by Proposition
\ref{prop2.3}. Hence over $M$, we have an induced linear connection on
any weighted tangent bundle, which we also denote by $\nabla$, and we obtain the following result.

\begin{thm}\label{thm2.7}
For $\barm=M\cup\partial M$, let $\nabla$ be a projectively pre-compact linear
connection on $TM$. Then for $\xi\in\frak X(\barm)$ and
$\eta\in\Ga(T^b\barm(-1))\subset\Ga(T\barm(-1))$, the weighted vector field
$\nabla_\xi\eta$ defined on $M$ admits a smooth extension to the boundary.

If the boundary is strongly totally geodesic, then this extension lies
in $\frak X_\partial(\barm)$. Hence, under this assumption, $\nabla$
extends smoothly to a linear connection on the vector bundle
$T^b\barm(-1)\to\barm$.
\end{thm}
\begin{proof}
The usual formula for projective changes in \cite{BEG} shows that, over $M$, we can
write $\nabla_\xi\eta=\nabla^\rho_\xi\eta-\frac{d\rho(\eta)}{\rho}\xi$ (and the
weight $-1$ is essential for the term containing $d\rho(\xi)$ to drop out). Since
$\eta$ is tangent to the boundary along the boundary, both summands admit a smooth
extension to the boundary. Applying $d\rho$ to the right hand side, we get
$d\rho(\nabla^\rho_\xi\eta)-\frac{d\rho(\xi)d\rho(\eta)}{\rho}$ which, in the case of
strongly totally geodesic boundary, vanishes along the boundary (see Theorem \ref{thm2.4}).
\end{proof}

Hence we have obtained a vector bundle $T^b\barm(-1)\to \barm$ of rank $n+1$, which
is endowed with a linear connection that we still denote by $\nabla$. This bundle can
be restricted to $\partial M$ without problems, and it follows from the general
theory of linear connections that one can also restrict the connection to the
boundary. Indeed for $\xi\in\frak X(\partial M)$ and
$\eta\in\Ga(T^b\barm(-1)|_{\partial M})$ one chooses (local) extensions to
$\tilde\xi\in\frak X_{\partial}(\barm)$ and $\tilde\eta\in\Ga(T^b\barm(-1))$ and then
considers $\nabla_{\tilde\xi}\tilde\eta|_{\partial M}$. It is clear that this depends
only on $\tilde\xi|_{\partial M}=\xi$, while for $\eta$, the freedom is to add a
$\rho\tilde\eta_1$ for $\eta_1\in\Ga(T^b\barm(-1))$, and one immediately concludes
that this leads to the same restriction to the boundary.

\subsection{Relation to boundary tractors}\label{2.8}
Starting from a projectively pre-compact connection $\nabla$ with
strongly totally geodesic boundary, we get an induced projective
structure on the boundary $\partial M$ as well as the rank $n+1$
vector bundle $T^b\barm(-1)|_{\partial M}\to\partial M$ which is
endowed with a linear connection induced by $\nabla$. One may hope
that this is related to the projective standard tractor bundle of the
boundary projective structure. There is a relation without additional
assumptions, but the main result is that under a slightly stronger
condition, the restriction indeed coincides with the standard tractor
bundle and its normal tractor connection.
We want to stress here that this conceptual description of
the boundary tractors is new even in the case of projectively compact
connections, the results in \cite{Proj-comp} apply only in the case
that $\nabla$ is Ricci flat in a neighborhood of the boundary. In
contrast, we only require asymptotic Ricci-flatness in directions
tangent to the boundary here.

\smallskip

Let us briefly recall the standard approach to projective tractors on a manifold $N$
of dimension $n$ endowed with a projective class of affine connections, see
\cite{BEG} for more details  and explicit formulae. One usually starts from the standard cotractor bundle,
which can be simply defined as a first jet prolongation of a line bundle $\Cal
T^*N:=J^1\Cal E(1)$. Here $\Cal E(1)\to N$ is the bundle of (projective)
$1$-densities, a trivial line bundle characterized by the fact that $\Cal
E(1)^{\otimes^{2n+2}}\cong(\Lambda^nTN)^{\otimes^2}$. The jet exact sequence for this
bundle has the form
$$
0\to T^*N(1)\to \Cal T^*N\to \Cal E(1)\to 0.  
$$
Any linear connection on $\Cal E(1)$ defines a splitting of this short exact
sequence, and in particular, one can apply this to the connections induced by the
affine connections in the projective class. Hence any connection in the projective
class gives rise to an identification $\Cal T^*N\cong \Cal E(1)\oplus T^*N(1)$ and
one can use the connection defining the identification, its Schouten tensor and
algebraic ingredients to define a linear connection on $\Cal T^*N$ in the given
splitting and then verify that the construction is independent of the choice of
connection. Hence it gives rise to a canonical connection on $\Cal T^*N$ called the
\textit{standard cotractor connection}.

The standard tractor bundle $\Cal TN\to N$ is then defined as the dual bundle to
$\Cal T^*N$, which implies that one gets an exact sequence
\begin{equation}\label{std-sequ}
0\to \Cal E(-1)\to \Cal TN\to TN(-1)\to 0 
\end{equation}
and any connection in the projective class determines a a splitting of this sequence.
The standard tractor connection on $\Cal TN$ is then obtained from dualizing the
standard cotractor connection and it is easy to derive the formula for the standard
tractor connection in such a splitting, see \cite{BEG}.

\begin{thm}\label{thm2.8}
  For $\barm=M\cup\partial M$, let $\nabla$ be an affine connection on
  $M$, which is projectively pre-compact and let $\Rho$ be the
  Schouten tensor of $\nabla$. Suppose further that for all
  $\xi,\eta\in\frak X(\barm)$, the function $\Rho(\xi,\eta)$ admits a
  smooth extension to the boundary if at least one of the fields lies
  in $\frak X_{\partial}(\barm)$ and its boundary value vanishes if
  both fields lie in $\frak X_{\partial}(\barm)$. Then the bundle
  $T^b\barm(-1)|_{\partial M}$ with the connection induced by $\nabla$
  can be naturally identified with the projective standard tractor
  bundle of the induced projective structure on $\partial M$ and its
  normal tractor connection.
\end{thm}
\begin{proof}
  First note that our assumptions imply that $\nabla$ has strongly totally geodesic
  boundary, so we indeed get a linear connection on $T^b\barm(-1)|_{\partial M}$
  induced by $\nabla$. We next observe that by expression \eqref{Tb-exact} from
  Lemma \ref{lem2.6a}, the bundle $T^b\barm(-1)|_{\partial M}$ fits into an exact
  sequence which looks like the sequence \eqref{std-sequ} for the standard tractor
  bundle. The only difference is that the $(-1)$-density bundle of $\partial M$ is
  replaced by the restriction $\Cal E(-1)|_{\partial M}$ of the $(-1)$-density bundle
  of $\barm$. (Note that the power in the definition of $(-1)$-densities as a root of
  the bundle of volume forms depends on the dimension). So we first prove that there
  is a natural way to identify these two bundles.

  Fix a non-zero density $\si\in\Ga(\Cal E(1)|_M)$ which is parallel
  for $\nabla$. By Proposition \ref{prop2.3}, this extends by $0$ to a
  smooth section of $\Cal E(1)\to\barm$ which is a defining density
  for $\partial M$. Taking any linear connection $\tilde\nabla$ on
  $\Cal E(1)$, the derivative $\tilde\nabla\si|_{\partial M}$ is a
  nowhere vanishing section of $T^*\partial M\otimes\Cal
  E(1)|_{\partial M}$. This is independent of the choice of
  $\tilde\nabla$ and its kernel in a point $x\in\partial M$ coincides
  with $T_x\partial M\subset T_x\barm$. Putting $n:=\dim(\partial M)$,
  taking the wedge product with this weighted one form induces an
  isomorphism $\Lambda^nT^*\partial M\to\Lambda^{n+1}T^*M\otimes\Cal
  E(1)$, and hence between the squares of this line bundles. But the
  square of the first bundle is the (intrinsic) density bundle $\Cal
  E_{\partial M}(-2n-2)$ while square of the second bundle is the
  restriction to $\partial M$ of $\Cal E(-2n-4)\otimes\Cal E(2)=\Cal
  E(-2n-2)$. Since this extends to powers, roots and duals, we can in
  particular identify $\Cal E(1)|_{\partial M}$ with the $\Cal
  E_{\partial M}(1)$, so we drop the subscript.

  Thus we have an exact sequence of the right form. In order to obtain the usual
  formula for the standard tractor connection, we have to make a small change, 
  however. We define the inclusion $\Cal E(-1)\hookrightarrow T^b\barm(-1)|_{\partial
    M}$ by sending a density $\tau$ to the boundary value of $-\tau\rho\nu$ where
  $d\rho(\nu)|_{\partial M}\equiv 1$.  From the proof of Lemma \ref{lem2.6a} we also
  see how a choice of local defining function $\rho$ gives rise to a splitting of the
  exact sequence. Given $\tilde\eta\in\Ga(T^b\barm(-1))$ the $(-1)$-density
  $-\frac{d\rho(\tilde\eta)}{\rho}$ admits a smooth extension to the boundary and its
  boundary value depends only on the boundary value of $\tilde\eta$, which is a
  section of $T^b\barm(-1)|_{\partial M}$. Hence we have constructed a bundle map
  $T^b\barm(-1)|_{\partial M}\to\Cal E(-1)$ (depending on the choice of $\rho$),
  which evidently splits the inclusion defined above.

  So it remains to compute the connection induced by $\nabla$ in such a splitting and
  we fix the defining function $\rho$ to determine the splitting. Then we take
  $\tilde\eta\in\Ga(T^b\barm(-1))$ and $\tilde\xi\in\frak X_{\partial}(\barm)$ and
  denote by $\xi\in\frak X(\partial M)$ the boundary value of the vector field
  $\tilde\xi$. Then we know that, in the splitting determined by $\rho$, the boundary
  values of $\tilde\eta$ corresponds to $(\eta,\tau)$, where $\eta$ is the boundary
  value of the vector field $\tilde\eta$ and $\tau$ is the boundary value of
  $-\frac{d\rho(\tilde\eta)}{\rho}$. Now we compute
  \begin{equation}\label{nabla}
    \nabla_{\tilde\xi}{\tilde\eta}=\nabla^\rho_{\tilde\xi}{\tilde\eta}-
    \frac{d\rho(\tilde\eta)}{\rho}\tilde\xi,
  \end{equation}
  and we have to determine the boundary value of this section of $T^b\barm(-1)$ in
  the splitting determined by $\rho$. In the first component, this simply is the
  boundary value as a vector field, which is just $\nabla^{\rho}_{\xi}\eta+\tau\xi$,
  where $\nabla^{\rho}$ also denotes the connection on $T\partial M(-1)$ induced by
  the connection $\nabla^{\rho}$ in the projective class. This agrees with the formula
  for the standard tractor connection in Section 3.2 of \cite{BEG}.

  To deal with the other component, we need one more observation. For $\tilde\xi$ and
  another vector field $\tilde\ze\in\frak X_{\partial}(\barm)$, we can apply the formula
  \eqref{fund-form2}, taking into account that $\Rho(\tilde\xi,\tilde\eta)=\Cal
  O(\rho)$ by assumption, to conclude that
  $$
  d\rho(\nabla^\rho_{\tilde\xi}\tilde\zeta)=\tilde\xi\cdot d\rho(\tilde\ze)+
  \rho\Rho^\rho(\tilde\xi,\tilde\zeta)+\Cal O(\rho^2).
  $$
  This equation can be twisted by a density without problems, one only has to replace
  the derivative in the first term of the right hand side by
  $\nabla^\rho_{\tilde\xi}$, so
$$
d\rho(\nabla^\rho_{\tilde\xi}\tilde\eta)=\nabla_{\tilde\xi}d\rho(\tilde\eta)+
\rho\Rho^\rho(\tilde\xi,\tilde\eta)+\Cal O(\rho^2). 
$$ Taking this into account, we see that applying
$\frac{-d\rho}{\rho}$ to \eqref{nabla} and ignoring terms that are
  $\Cal O(\rho)$, we obtain 
$$
-\frac{1}{\rho}\nabla^\rho_{\tilde\xi}d\rho(\tilde\eta)-\Rho^\rho(\tilde\xi,\tilde\eta)+
\frac{1}{\rho^2}d\rho(\tilde\xi)d\rho(\tilde\eta),
$$
for the right hand side.
The first and last terms here add up to
$-\nabla^\rho_{\tilde\xi}\frac{d\rho(\tilde\eta)}{\rho}$, so the boundary value of
the whole expression is $\nabla_\xi\tau-\Rho^\rho(\xi,\eta)$. This matches the second
component of the formula for the tractor connection in \cite{BEG}, apart from the
fact that $\Rho^\rho$ is the boundary value of the Schouten tensor of the connection
$\nabla^\rho$ on $\barm$ rather then the Schouten tensor of the induced connection on
the boundary, so it remains to show that these agree.

 Since the boundary is totally geodesic, the curvature of the
  restriction of $\nabla^\rho$ to $\partial M$ is obtained by inserting vector fields
  in $\frak X_\partial (\barm)$ into the curvature of $\nabla^\rho$ and then taking
  boundary values. However, passing to the Ricci curvature, one has to take traces
over different spaces. Fixing $\rho$, we choose a local vector field $\nu$ such that
$d\rho(\nu)\equiv 1$. Then we choose a local frame for $T\partial M$ along $\partial
M$ and extend the elements to vector fields which insert trivially into $d\rho$
to obtain a local smooth frame for $T\barm$. This has the property that the dual
coframe consists of $d\rho$ and of forms that restrict to the dual coframe of the
initial frame for $T\partial M$. For vector fields $\tilde\xi,\tilde\zeta\in\frak
X_{\partial}(\barm)$ we denote by $\xi,\zeta\in\frak X(\partial M)$ their
restrictions to $\partial M$. Denoting by $\Ric^\rho_\partial$ the Ricci curvature of
the boundary connection and by $\Ric^\rho$ the boundary value of the Ricci curvature
of the interior connection, the difference
$\Ric^\rho(\xi,\zeta)-\Ric^\rho_\partial(\xi,\zeta)$ clearly is the boundary value of
$d\rho(R(\nu,\tilde\xi)(\tilde\zeta))$.

Now we can apply $d\rho$ to the defining equation for
$R(\nu,\tilde\xi)(\tilde\zeta)$, then use \eqref{fund-form2} and drop terms that are
$\Cal O(\rho)$ to get
$$
\nu\cdot d\rho(\nabla^\rho_{\tilde\xi}\tilde\zeta)-\tilde\xi\cdot
d\rho(\nabla^\rho_\nu\tilde\zeta)-[\nu,\tilde\xi]\cdot d\rho(\tilde\zeta). 
$$
Next, for the first two summands, we can again use \eqref{fund-form2}, but have to carry
out the differentiation before dropping $\Cal O(\rho)$-terms. Up to terms that are
obviously $\Cal O(\rho)$, this gives
\begin{multline*}
\nu\cdot\tilde\xi\cdot d\rho(\tilde\zeta)+
d\rho(\nu)(\Rho^\rho(\tilde\xi,\tilde\ze)-\Rho(\tilde\xi,\tilde\ze))\\-
\tilde\xi\cdot\nu\cdot d\rho(\tilde\zeta)-
d\rho(\tilde\xi)(\Rho^\rho(\nu,\tilde\ze)-\Rho(\nu,\tilde\ze))-
[\nu,\tilde\xi]\cdot d\rho(\tilde\zeta). 
\end{multline*}
Of course the three derivative terms cancel by definition of the Lie bracket, and
$d\rho(\tilde\xi)$ is $\Cal O(\rho)$, so this part can be dropped, too. Since
$d\rho(\nu)\equiv 1$ and by our assumption $\Rho(\tilde\xi,\tilde\ze)$ vanishes along
the boundary, we conclude that
$$
\Ric^\rho(\xi,\zeta)-\Ric^\rho_\partial(\xi,\zeta)=\Rho^\rho(\xi,\zeta)=\tfrac{1}{n}\Ric^\rho(\xi,\zeta)
$$
and hence $\Ric^\rho_\partial(\xi,\zeta)=\frac{n-1}{n}\Ric^\rho(\xi,\zeta)$ which
exactly implies that the resulting Schouten tensors agree. 
\end{proof}

\subsection{Relation to interior tractors}\label{2.9} While considerations about
tractors inspired many of  the developments in this article, it turned out that
technically the weighted b-tangent bundle offers a much simpler
approach that leads to stronger results. Thus we will keep the
discussion of tractors rather short, but we believe that the relation
of tractors to the b-tangent bundle is of independent interest.

For $\barm=M\cup\partial M$, a projectively pre-compact affine
connection $\nabla$ on $M$ defines a projective structure on $M$,
which makes tools like tractor bundles and tractor connections
available. As mentioned in Section \ref{2.3}, one cannot expect that
this projective structure admits a smooth extension to all of
$\barm$. However, we have already seen there that some of these tools
\textit{do} extend to the boundary. In particular, the standard
cotractor bundle can be defined as $J^1\Cal E(1)$ on all of $\barm$
and the jet exact sequence shows that it has the right composition
structure on all of $\barm$. Moreover, by Lemma \ref{lem2.4}, for the
linear connection $\nabla^\rho$ on $TM$ associated to a defining
function $\rho$, the induced connection on $\Cal E(1)$ extends
smoothly to the boundary. Since such a connection splits the jet exact
sequence, also the splittings of the standard cotractor bundle coming
from a defining function are smooth up to the boundary.

Since general tractor bundles can be obtained by tensorial constructions from the
standard cotractor bundle, all these bundles and their composition series and
splittings extend to all of $\barm$. In particular, we obtain an analog of the
standard tractor bundle on all of $\barm$ as $\Cal T\barm:=(J^1\Cal E(1))^*$ and this
comes with a short exact sequence
$$
0\to\Cal E(-1)\to \Cal T\barm \to T\barm(-1)\to 0
$$
and a splitting of this sequence associated to any defining function $\rho$. In
particular, any section of $\Cal T\barm$ projects to a weighted vector field of
weight $-1$. Hence we can define a subsheaf $\Ga_{\partial}(\Cal
T\barm)\subset\Ga(\Cal T\barm)$ as by requiring that this projection is tangent to
$\partial M$ along $\partial M$. As for the tangent bundle, this gives rise to a
b-tractor bundle $\Cal T^b\barm$, which comes with an exact sequence as above, but
with $T^b\barm(-1)$ replacing $T\barm(-1)$.

Next, we can consider the formula for the standard tractor connection
$\nabla^{\Cal T}$ in the splittings determined by a defining function,
which is mentioned in the proof of Theorem \ref{thm2.8}, see
\cite{BEG} for details. Apart from tensorial operations, this only
involves the connection $\nabla^\rho$ and its Schouten tensor
$\Rho^{\rho}$. Assuming that $\nabla$ has totally geodesic boundary,
we can apply this formula over $M$, to conclude that for $\xi\in\frak
X(\barm)$ and $s\in\Ga(\Cal T\barm)$, the section $\nabla^{\Cal T}_\xi
s$ of $\Cal T\barm|_M$ admits a smooth extension to the boundary if
$\xi\in\frak X_{\partial}(\barm)$ or $s\in\Ga(\Cal T^b\barm)$. Hence
the tractor connection restricts to well defined operations on $\frak
X_{\partial}(\barm)\x\Ga(\Cal T\barm)$ and $\frak X(\barm)\x\Ga(\Cal
T^b\barm)$. However, for the second operation, the values do not lie
in $\Ga(\Cal T^b\barm)$, so it seems that one does \textit{not} get a
connection on $\Cal T^b\barm$.

As we have seen in Proposition \ref{prop2.3}, a nonzero section $\si\in\Ga(\Cal
E(1)|_M)$ that is parallel for $\nabla$ extends by zero to a defining density for
$\partial M$ and hence determines a global, nowhere vanishing section $I=j^1\si$ of
$J^1(\Cal E(1))$. This spans a smooth line subbundle in $J^1(\Cal E)$, whose
annihilator is a smooth hyperplane bundle $I^\perp\subset\Cal T\barm$. Since $\si$ is
a defining density, it follows readily that any section of $I^\perp$ projects onto a
weighted vector field that is tangent to $\partial M$ along $\partial M$. This shows
that we can actually view $I^\perp$ as a subbundle of $\Cal T^b\barm$.

The applications of tractors in \cite{Proj-comp} (where the projective structure
admits a smooth extension to the boundary), which are relevant here, require that $I$
is parallel for the tractor connection. This is equivalent to $\nabla$ being Ricci
flat, so in particular, the boundary is strongly totally geodesic. Since $I$ is
parallel, also the subbundle $I^\perp$ is parallel and hence inherits a connection,
and it turns out that $I^\perp$ and this connection restrict to the projective
standard tractor bundle on $\partial M$ and its normal tractor connection. This
approach can be generalized directly to the projectively pre-compact case via the
operator $\frak X_{\partial}(\barm)\x\Ga(J^1\Cal E(1))\to\Ga(J^1\Cal E(1))$ obtained
by dualizing the corresponding operation on $\Cal T\barm$.

However, this can also be easily deduced from Theorem \ref{thm2.8} (which works under
much weaker assumptions) as follows. One easily verifies that the projection $\Cal
T^b\barm\to T^b\barm(-1)$ restricts to an isomorphism $I^\perp\to T^b\barm(-1)$,
which over $M$ is just the inverse of the splitting induced by the connection
$\nabla$. Moreover, if $\nabla$ is Ricci-flat then, in this splitting, the normal
tractor connection is simply given by $\nabla$, which completes the argument. 

\section{Projectively pre-compact metrics}\label{3}
We now turn to the case of projectively pre-compact pseudo-Riemannian metrics, which
is of interest for applications to general relativity. The main topic here is the
 relation of projective pre-compactness to asymptotic forms of the
  metric. One of these forms is only available locally around boundary points in a
  dense open subset and plays an important role in general relativity. The second
  asymptotic form is available everywhere and to our knowledge is new.

\subsection{An asymptotic form}\label{3.1}
In our usual setting $\barm=M\cup\partial M$, we want to study a certain asymptotic
form for a pseudo-Riemannian metric $g$ on $M$. The basic setup is that for some
function $C$, with properties to be specified later, and an appropriately chosen local
defining function $\rho$, we assume that
\begin{equation}
  \label{h-def}
h:=\rho^2(g-\frac{C}{\rho^2}d\rho^2)  
\end{equation}
admits a smooth extension to the boundary such that the restriction of the boundary
value to $T\partial M\x T\partial M$ is non-degenerate. This means that then $g$
has the asymptotic form
\begin{equation}
  \label{g-asymp}
  g=\frac{C}{\rho^4}d\rho^2+\frac{h}{\rho^2}. 
\end{equation}
Observe that such a form \textit{does} depend on the choice of the defining function
$\rho$. For $\rho=e^{-f}\hat\rho$ we get $d\rho=e^{-f}(d\hat\rho-\rho df)$ and
inserting this, we see that we will get terms involving the symmetric product of
$d\hat\rho$ and $df$, which are $\Cal O(\rho^{-3})$ and cannot be absorbed into $C$
unless $df=d\hat\rho+\Cal O(\rho)$. On the other hand, a constant rescaling of $\rho$
leads to a constant rescaling of $h$ and $C$, so one may absorb constant multiples of
$C$ into a change of defining function.

The basic assumption will be that $C$ equals a non-zero constant along
the boundary, and this constant can be assumed to be $\pm
1$. Moreover, it has been shown in \cite{Proj-comp} that if $C=\pm
1+\Cal O(\rho^2)$ and the asymptotic form \eqref{g-asymp} is available
locally around each boundary point, then $g$ is projectively compact
of order $1$. It has also been shown in \cite{Proj-comp}
that for any pseudo-Riemannian metric $g$ which is projectively
compact of order $1$, such an asymptotic form is available locally
around appropriate parts of the boundary, which form a dense open
subset of the boundary. For applications in general relativity it turns out to
be of fundamental importance to allow that $C=\pm 1+\Cal O(\rho)$ in
order to allow for  a non-trivial mass aspect, see
\cite{Ashtekar-Romano}.  We now show that such weaker fall-off conditions on
the non-constant part of $C$ imply projective pre-compactness.

\begin{thm}\label{thm3.1}
  For $\barm=M\cup\partial M$, suppose that $g$ is a pseudo-Riemannian
  metric on $M$ such that locally around each $x\in\partial M$, there
  is a defining function $\rho$ and a nowhere vanishing function $C$
  such that the section $h$ defined in \eqref{h-def} admits a smooth
  extension to the boundary, whose boundary value is non-degenerate on
  $T\partial M\x T\partial M$. Suppose further that the function $C$
  has the property that for each vector field $\ze$ that is smooth up
  to the boundary and such that $d\rho(\ze)\equiv 0$, we have
  $\ze\cdot C=O(\rho)$. Then $g$ is projectively pre-compact with
  totally geodesic boundary. If we even have $\ze\cdot C=\Cal
  O(\rho^2)$ for any $\ze$ as above, then $g$ is projectively compact
  (of order $1$).
\end{thm}
\begin{proof}
  Apart from the part asserting that the boundary is totally geodesic,
  this follows the analogous proof of Theorem 2.6 of \cite{Proj-comp},
  which, in particular, proves the last statement and actually implies
  parts of our more general result. Nevertheless, we repeat parts of
  that proof to ensure that the stronger assumptions on $C$ imposed in
  \cite{Proj-comp} were not used implicitly. First
  observe that it suffices to verify the defining conditions for
  projective pre-compactness for the connection $\nabla^\rho$
  associated to the defining function for which the asymptotic form is
  available. In particular, we have the tensor field $h$ at our
  disposal, and restricting to a sufficiently small neighborhood of
  the boundary, we may assume that $h$ is non-degenerate on
  $\ker(d\rho)\x\ker(d\rho)$ everywhere. As in \cite{Proj-comp}, this
  implies that there is an analog of a Reeb field,
    i.e.\ a vector field $\ze_0$ such that $d\rho(\ze_0)\equiv 1$ and
  such that $h(\ze_0,\ze)\equiv 0$ for $\ze\in\ker(d\rho)$.

  We can extend $\ze_0$ to a local frame by vector fields $\ze_i$ in $\ker(d\rho)$
  and then consider the frame $\tilde\ze_i$ for $\ker(d\rho)$ dual to $\{\ze_i\}$
   with respect to $h|_{\ker(d\rho)}$. For a vector field $\xi$, we
  then get $\xi=d\rho(\xi)\ze_0+\sum_ih(\xi,\tilde\ze_i)\ze_i$, so to prove that
  $\xi$ admits a smooth extension to the boundary, it suffices to prove that the
  functions $d\rho(\xi)$ and $h(\xi,\ze)$ for $\ze\in\ker(d\rho)$ admit smooth
  extensions.

  Now the asymptotic form \eqref{g-asymp} readily implies that on $M$, we get
  $$
  g(\eta,\ze_0)=d\rho(\eta)\left(\frac{C}{\rho^4}+\frac{h(\ze_0,\ze_0)}{\rho^2}\right),   
  $$
  so to show that $d\rho(\xi)$ admits a smooth extension, it suffices to prove that
  $\rho^4g(\xi,\ze_0)$ is smooth up to the boundary. Likewise, if $\ze$ lies in
  $\ker(d\rho)$, the asymptotic form shows that in the interior
  $g(\xi,\ze)=\frac{1}{\rho^2}h(\xi,\ze)$, so to show that $h(\xi,\ze)$ is smooth up
  to the boundary, it suffices to prove that $\rho^2g(\xi,\ze)$ has this property.

  To apply this to a vector field of the form $\nabla^\rho_\xi\eta$, one uses the
  Koszul formula to express $\nabla_\xi\eta$ to conclude that, on $M$,
  $2g(\nabla^\rho_\xi\eta,\ze)$ can, for any vector field $\ze$, be written as
  \begin{equation}    \label{Koszul}
    \begin{aligned}
     &\xi\cdot g(\eta,\zeta)-\ze\cdot g(\xi,\eta)+\eta\cdot
     g(\xi,\ze)\\ +&g([\xi,\eta],\ze)-g([\xi,\ze],\eta)-g([\eta,\ze],\xi)\\
     +&\tfrac{2}{\rho}d\rho(\xi)g(\eta,\ze)+\tfrac{2}{\rho}d\rho(\eta)g(\xi,\ze). 
    \end{aligned}
  \end{equation}
  
  Now the asymptotic form for $g$ readily implies that we can write $\xi\cdot
  g(\eta,\zeta)$ as
  $$
  -4\rho^{-5}Cd\rho(\xi)d\rho(\eta)d\rho(\zeta)+\rho^{-4}\xi\cdot(Cd\rho(\eta)d\rho(\zeta))+\Cal
  O(\rho^{-3}). 
  $$
  Putting $\ze=\ze_0$ and applying this to all three terms in the first line in
  \eqref{Koszul}, the $\rho^{-5}$-terms cancel with the last line of \eqref{Koszul}
  and we can write \eqref{Koszul} in this case as
\begin{equation}    \label{Koszul2}
  \begin{aligned}
    \rho^{-4}&\left(\xi\cdot
    (Cd\rho(\eta))-\ze_0\cdot(Cd\rho(\xi)d\rho(\eta)+\eta\cdot (Cd\rho(\xi))\right)+\\
	  +&g([\xi,\eta],\ze)-g([\xi,\ze],\eta)-g([\eta,\ze],\xi)+\Cal O(\rho^{-3}).
    \end{aligned}
\end{equation}
Obviously, all terms in this formula are $\Cal O(\rho^{-4})$, so we conclude
$d\rho(\nabla^\rho_{\xi}\eta)$ always admits a smooth extension to the boundary
(regardless of any fall off conditions on derivatives of $C$).

Assuming, in addition, that $\xi,\eta\in\frak X_{\partial}(\barm)$, we have
$Cd\rho(\eta)=\Cal O(\rho)$ and this remains true after differentiation
by $\xi$ and similarly for the last term in the first line of
\eqref{Koszul2}. Likewise, $Cd\rho(\xi)d\rho(\eta)=\Cal O(\rho^2)$ and
differentiating in the direction $\ze_0$, the result is $\Cal O(\rho)$, so
the whole first line of \eqref{Koszul2} is $\Cal O(\rho^{-3})$ in this
case. For $\xi,\eta\in\frak X_\partial(\barm)$, we get
$[\xi,\eta]\in\frak X_\partial(\barm)$ and so, in each of the terms in the
second line of \eqref{Koszul2}, one of the vector fields that is
inserted into $g$ lies in $X_\partial(\barm)$. By the asymptotic form
of $g$, all these terms thus also are $\Cal O(\rho^{-3})$, so for
$\xi,\eta\in\frak X_\partial(\barm)$ we see that
$d\rho(\nabla^{\rho}_\xi\eta)$ is $\Cal O(\rho)$, so the boundary is
totally geodesic.

  We next analyze \eqref{Koszul} in the case that $d\rho(\ze)\equiv 0$. Then
  $g(\eta,\ze)=\frac1{\rho^2}h(\eta,\ze)$ for any $\ze\in\frak X(\barm)$ and
  differentiating this in the direction $\xi$, we obtain
  $$
  -2\rho^{-3}d\rho(\xi)h(\eta,\zeta)+\Cal O(\rho^{-2}).
  $$
  Applying this to the first
  and the last term in the first line of \eqref{Koszul}, there is again a cancellation
  with the terms in the last line, while the first summand in the second line is
  $\Cal O(\rho^{-2})$. Overall, we conclude that, for the case of $d\rho(\ze)\equiv 0$, 
  \eqref{Koszul} can be written as
  \begin{equation}\label{Koszul3}
-\ze\cdot g(\xi,\eta)-g([\xi,\zeta],\eta)-g([\eta,\zeta],\xi)+\Cal O(\rho^{-2}). 
  \end{equation}
  Now $g([\xi,\ze],\eta)=\tfrac{C}{\rho^4}d\rho([\xi,\ze])d\rho(\eta)+\Cal
  O(\rho^{-2})$, and $dd\rho(\xi,\ze)=0$
  and $d\rho(\ze)=0$ imply $d\rho([\xi,\ze])=-\ze\cdot d\rho(\xi)$, and like-wise for
  the term with $\xi$ and $\eta$ exchanged. On the other hand, the
  first term in \eqref{Koszul3} gives
  $$
  -\ze\cdot\left(\rho^{-4}Cd\rho(\xi)d\rho(\eta)+\Cal O(\rho^{-2})\right).
  $$
  By assumption $\ze\cdot\rho\equiv 0$, so we may pull out the factor $\rho^{-4}$ in
  the first summand, while the second summand will remain to be $\Cal O(\rho^{-2})$
  after differentiation by $\ze$. Moreover, the the terms in which $\ze$
  differentiates $d\rho(\xi)$ respectively $d\rho(\eta)$ cancel with the contributions
  from above. Hence we conclude that if $d\rho(\ze)\equiv 0$, then  without any
  assumption on $\xi$ and $\eta$, \eqref{Koszul3} equals
  
  $$
-\rho^{-4}(\ze\cdot C)d\rho(\xi)d\rho(\eta)+\Cal O(\rho^{-2}). 
$$
Now by assumption $\ze\cdot C=\Cal O(\rho)$, so this sum is always $\Cal
O(\rho^{-3})$ and it is $\Cal O(\rho^{-2})$ if either $\xi$ or $\eta$ lies in $\frak
X_{\partial}(\barm)$. Hence $h(\rho\nabla^\rho_\xi\eta,\ze)=\rho
h(\nabla^\rho_\xi\eta,\ze)$ is smooth up to the boundary and if either $\xi$ or
$\eta$ lies in $\frak X_\partial(\barm)$ then $h(\nabla^\rho_\xi\eta,\ze)$ is smooth
up to the boundary, which completes the proof that $\nabla^\rho$ is projectively
pre-compact.

If we even have $\ze\cdot C=\Cal O(\rho^2)$, then the whole expression is $\Cal
O(\rho^{-2})$ for any $\xi$ and $\eta$, which shows that $h(\nabla^\rho_\xi\eta,\ze)$
is always smooth up to the boundary and hence $g$ is projectively compact.
\end{proof}

\subsection{The boundary structure}\label{3.2}
For $\barm=M\cup\partial M$ consider a pseudo-Riemannian metric $g$ on $M$ which
admits an asymptotic form \eqref{g-asymp} and let $\Ric$ be the Ricci curvature of
$g$, which is just a constant multiple of the Schouten tensor of the Levi-Civita
connection $\nabla$ of $g$. Then by Theorems \ref{thm3.1} and \ref{thm2.6} and
Corollary \ref{cor2.6}, we know that $\rho\Ric$ admits a smooth extension to the
boundary and $\Ric(\xi,\eta)$ itself extends smoothly if both $\xi$ and $\eta$ lie
in $\frak X_{\partial}(\barm)$.

To proceed towards a converse of Theorem \ref{thm3.1}, we have to
impose a slightly stronger assumption on the Ricci curvature of $g$ to
be able to apply the results form Section \ref{2}. Our standing
assumption from now on will be the following.

\begin{definition}\label{def3.2}
For $\barm=M\cup\partial M$, we call a projectively pre-compact pseudo-Riemannian
metric $g$ on $M$ with Ricci curvature $\Ric$ \textit{asymptotically tangentially
  Ricci flat} if and only if $\Ric(\xi,\eta)$ admits a smooth extension to the
boundary if one of the fields lies in $\frak X_\partial(\barm)$ and vanishes along
the boundary if both fields lie in $\frak X_{\partial}(\barm)$.
\end{definition}
Observe that this property implies that the boundary is strongly totally geodesic, so
we can apply all the developments of Section \ref{2} to the Levi-Civita connection
$\nabla$ of $g$. Our first result provides a strong additional geometric structure on
the boundary, which, even in the projectively compact case, was only known under the
assumption that $g$ is Ricci-flat. Note that as a density $\si\in\Ga(\Cal E(1))$
that is parallel for $\nabla$ over $M$, we can take the appropriate power of the
volume density of $g$ here. We call this the {\em volume scale}.

\begin{thm}\label{thm3.2}
For $\barm=M\cup\partial M$ let $g$ be a projectively pre-compact
pseudo-Riemannian metric of signature $(p,q)$ on $M$, which is
asymptotically tangentially Ricci flat, and let $\si\in\Ga(\Cal E(1))$
denote the volume scale.
Then the smooth section
$\si^2g\in\Ga(S^2T^*M(2))$ admits a smooth extension to a section of
$S^2(T^b\barm)^*(2)$ over all of $\barm$.

Moreover, the boundary value of this extension is parallel for the connection induced
by $\nabla$ and non-degenerate as a bundle metric on $T^b\barm(-1)|_{\partial
  M}$. Hence it provides a reduction of the projective holonomy of the projective
structure on $\partial M$ obtained in Proposition \ref{prop2.5} to the group
$SO(p,q)\subset SL(n+1,\Bbb R)$.
\end{thm}
\begin{proof}
Since $g\in\Ga(S^2T^*M)$ and $\si\in\Ga(\Cal E(1))$ are parallel for the connection
induced by $\nabla$, so is $\si^2g\in\Ga(S^2T^*M(2))$. But we can view this as a
section of $S^2(T^b\barm)^*(2)$ defined over $M$. By Theorem \ref{thm2.7}, $\nabla$
smoothly extends to a linear connection on $T^b\barm(-1)$ which which in turn induces
a linear connection on the symmetric square of the dual, i.e.~on
$S^2(T^b\barm)^*(2)$. Of course, $\si^2g$ is parallel for this connection over $M$,
so it extends by parallel transport to a smooth section of $S^2(T^b\barm)^*(2)$ over
all of $\barm$.

Since our connection is induced by a linear connection on
$T^b\barm(-1)$, the non-degeneracy of $\si^2g$ as a bundle metric on
$T^bM(-1)$ continues to hold for the extension. (View the section as a
bundle map $T^b\barm(-1)\to (T^b\barm(-1))^*$. If the extension would
be non-injective in some boundary point, one could parallely transport
a non-zero element in the kernel to obtain a contradiction to
injectivity in the interior. Indeed the signature is also preserved.)
Hence the boundary value is a non-degenerate bundle metric on
$T^b\barm(-1)|_{\partial M}$, which is parallel for the connection
induced by $\nabla$. Under the identification from Theorem
\ref{thm2.8}, this corresponds to a bundle metric on the projective
standard tractor bundle which is parallel for the normal tractor
connection and hence provides the claimed holonomy reduction.
\end{proof}

Via the general theory of holonomy reductions of Cartan connections
developed in \cite{hol-red}, this has remarkable consequences. The
case at hand has been discussed in detail in the literature, see
Section 3.1 of \cite{hol-red}, Section 3.2 of \cite{ageom} and Section
3.5 of \cite{Proj-comp}. The holonomy reduction decomposes $\partial
M$ into a disjoint union as $\partial M=\partial M_+\sqcup\partial
M_0\sqcup\partial M_-$ of so-called \textit{curved orbits}. Here
$\partial M_{\pm}$ are open subsets of $\partial M$, while $\partial
M_0$ (if non-empty) is an embedded hypersurface which separates these
two open subsets. On $\partial M_{\pm}$ the holonomy reduction
determines pseudo-Riemannian Einstein metrics of signature $(p-1,q)$
and $(p,q-1)$ and positive and negative scalar curvature, respectively, whose
Levi-Civita connections lie in (the 
restrictions of) the projective class. On $\partial M_0$, one obtains
a conformal structure of signature $(p-1,q-1)$ for which the
restriction of the tractor bundle provides the conformal standard
tractor bundle. Making the decomposition into curved orbits explicit
in our situation is easy.

\begin{cor}\label{cor3.2}
  For $\barm=M\cup\partial M$, let $g$ be a pseudo-Riemannian metric
  on $M$, which is projectively pre-compact and asymptotically
  tangentially Ricci-flat. For a defining function $\rho$, choose a
  vector field $\nu$ such that $d\rho(\nu)|_{\partial M}\equiv
  1$. Then the $2$-density $\rho^2\si^2g(\nu,\nu)$ defined over $M$
  admits a smooth extension to the boundary.

  The boundary value $\tilde C$ of this density is a smooth section of the oriented
  line bundle $\Cal E(2)\to\partial M$, which depends only on $g$ and not on the
  choice of $\rho$. Moreover, the decomposition $\partial M=\partial
  M_+\sqcup\partial M_0\sqcup\partial M_-$ is according to the strict sign of $\tilde{C}$ so
  in particular, $\partial M_0$ is the zero locus for $\tilde C$. Finally, $\tilde C$
  is a defining density for $\partial M_0$.
\end{cor}
\begin{proof}
As we have seen in Section \ref{2.6a}, the boundary value of
$\rho\nu\in\Ga(T^b\barm)$ is independent of the choice of $\rho$ and spans the kernel
of the projection $T^b\barm|_{\partial M}\to T\partial M$. It also defines the
inclusion $\Cal E(-1)|_{\partial M}\to T^b\barm(-1)|_{\partial M}$ which corresponds
to the inclusion of the canonical line bundle into the standard tractor bundle, see
Theorem \ref{thm2.8}. Hence the density $\tilde{C}$ exactly corresponds to the restriction of
the tractor metric to $S^2\Cal E(-1)\subset S^2\Cal T\partial M$. Now the result
follows from Theorem 3.1 of \cite{hol-red}.
\end{proof}

\subsection{Necessity of asymptotic forms}\label{3.3}
Now we are ready to prove that projectively pre-com\-pact metrics
which are asymptotically tangentially Ricci flat always admit specific
asymptotic forms. We first provide an asymptotic form that is valid
locally around any boundary point and for any defining function for
the boundary. This is new even in the case of Ricci-flat metrics that
are projectively compact of order one. We also derive a relation
between the boundary values of the coefficients in the
expansion. Finally, we prove that locally around boundary points that
lie in $\partial M_{\pm}$, one can choose specific defining functions
for which the asymptotic form specializes to
\eqref{g-asymp}.  To formulate this, we need a notation
  for the symmetric product of one-forms, so we put
  $(\ph_1\odot\ph_2)(\xi,\eta):=\frac12(\ph_1(\xi)\ph_2(\eta)+\ph_1(\eta)\ph_2(\xi))$.

\begin{thm}\label{thm3.3}
For $\barm=M\cup\partial M$, let $g$ be a pseudo-Riemannian metric of
  signature $(p,q)$ on $M$ which is
projectively pre-compact and asymptotically tangentially Ricci flat, and let $\rho$ be
any defining function for $\partial M$. 

(1) There is a smooth function $C:\barm\to\Bbb R$, a one-form $\la\in\Om^1(\barm)$
and a section $h\in\Ga(S^2T^*\barm)$ such that, over $M$, we have
\begin{equation}\label{g-asymp2}
g=\frac{Cd\rho^2}{\rho^4}+\frac{\la\odot d\rho}{\rho^3}+\frac{h}{\rho^2}. 
\end{equation}

(2) Given any asymptotic form as in \eqref{g-asymp2}, let $\underline{C}$ be the
boundary value of $C$, and let $\underline{\la}\in\Om^1(\partial M)$ and
$\underline{h}\in\Ga(S^2T^*\partial M)$ the restrictions of the boundary values of
$\la$ and $h$ to tangential entries. Then $\underline{C}$ is a defining function for
$\partial M_0$ and the symmetric bilinear form on the bundle $T\partial M\oplus\Bbb R$, over $\partial M$,
defined by
\begin{equation}\label{trac-met}
((\xi,a),(\eta,b)) \mapsto
  \underline{C}ab-\tfrac12 a\underline{\la}(\eta)-
            \tfrac12 b\underline{\la}(\xi)+\underline{h}(\xi,\eta)  
\end{equation}
is non degenerate of signature $(p,q)$. Moreover, denoting by $\nabla^{\rho}$ the
induced connection on $T\partial M$ and by $\Rho^\rho$ its Schouten tensor, we get
for $\xi,\eta,\zeta\in\frak X(\partial M)$
\begin{gather*}
  \underline{\la}=-dC \qquad
\underline{h}(\xi,\eta)=-\tfrac12 (\nabla^\rho_\xi\underline{\la})(\eta)+\underline{C}\Rho^\rho(\xi,\eta)\\
(\nabla^\rho_{\xi}\underline{h})(\eta,\zeta)=\tfrac12\big(
\Rho^\rho(\xi,\eta)\underline{\la}(\zeta)+
\Rho^\rho(\xi,\zeta)\underline{\la}(\eta) \big). 
\end{gather*}
\end{thm}
\begin{proof}
We treat (1) first.  Fix a vector field $\nu$ with $d\rho(\nu)\equiv 1$ and take the density $\tilde
C=\rho^2\si^2g(\nu,\nu)$ from Corollary \ref{cor3.2}, which extends smoothly to all
of $\barm$. Since $\frac{d\rho}{\rho}\in\Ga((T^b\barm)^*)$, we see that $\frac{\tilde
  C}{\rho^2}d\rho^2\in\Ga(S^2(T^b\barm)^*(2))$, so also $k:=\si^2g- \frac{\tilde
  C}{\rho^2}d\rho^2$ is a smooth section of this bundle. By construction we get
$k(\rho\nu,\rho\nu)\equiv0$. Putting $\tilde\la(\xi):= 2 k(\xi,\rho\nu)$ for
$\xi\in\frak X_{\partial}(\barm)$ defines a smooth section
$\tilde\la\in\Gamma((T^b\barm)^*(2))$ such that $\tilde\la(\rho\nu)\equiv 0$. Hence from
Section \ref{2.6a} we know that $\tilde\la$ actually lies in the subspace
$\Ga(T^*\barm(2))$. Then, $\frac{\tilde\la\odot d\rho}{\rho}$ defines a smooth
section of $S^2(T^b\barm)^*(2)$, hence also
\begin{equation}\label{hdef}
\tilde h:=k-\frac{\tilde\la\odot d\rho}{\rho}=\si^2g-\frac{\tilde C}{\rho^2}d\rho^2-
\frac1{\rho}\tilde\la\odot d\rho
\end{equation}
is a smooth section of $S^2(T^b\barm)^*(2)$.  The first (defining) equation shows that for
$\tilde\xi\in\frak X_{\partial}(\barm)$, we get $\tilde h(\xi,\rho\nu)=0$ and using
Section \ref{2.7} we conclude that $\tilde h\in\Ga(S^2T^*\barm(2))$.

Putting $\si^\rho=\frac{\si}{\rho}$, we obtain a section of $\Cal E(1)$ which is
nowhere vanishing and we define $C:=\frac{1}{(\si^\rho)^2}\tilde C$,
$\la=\frac{1}{(\si^\rho)^2}\tilde \la$ and $h:=\frac{1}{(\si^\rho)^2}\tilde h$. Of
course, $C\in C^\infty(\barm,\Bbb R)$, $\la\in\Om^1(\barm)$ and
$h\in\Ga(S^2T^*\barm)$ and multiplying \eqref{hdef} by
$\frac1{\si^2}=\frac{1}{\rho^2(\si^\rho)^2}$, we exactly get \eqref{g-asymp2}.

For part (2), we suppose now that we have any expression of the form \eqref{g-asymp2},
say with coefficients $\hat C$, $\hat\la$ and $\hat h$ and multiply  it by
$\si^2=\rho^2(\si^{\rho})^2$. Inserting two copies of $\rho\nu$ into the result, the
last two terms are $\Cal O(\rho)$ and $\Cal O(\rho^2)$, respectively. This shows that
$(\si^{\rho})^2\hat C+\Cal O(\rho)=\tilde C$ and hence $\hat C=C+\Cal O(\rho)$. In
particular, the boundary value of $\hat C$ equals $\underline{C}$. This also shows
that $\frac{(\si^\rho)^2}{\rho}\hat\la\odot d\rho+(\si^\rho)^2\tilde{h}=k+\Cal
O(\rho)$. Inserting $\rho\nu$ and $\xi\in\frak X_{\partial}(\barm)$ into this, we
conclude that $\hat\la(\xi)+\Cal O(\rho)=\la(\xi)$, so the boundary values of
$\hat\la$ and $\la$ agree on vectors tangent to the boundary. In the same way, one
concludes that the boundary values of $\hat h$ and $h$ agree if both their entries
are tangent to the boundary.

Hence it suffices to verify the remaining properties for the the asymptotic form we
have constructed. Since $C=\frac1{(\si^{\rho})^2}\tilde C$, and $\si^\rho$ is nowhere
vanishing, Corollary \ref{cor3.2} readily implies that $\underline{C}$ is a defining
function for $\partial M_0$. For the form we have constructed, \eqref{trac-met}
simply encodes the boundary value of $\frac{\si^2}{(\si^\rho)^2}g$ on
$T^b\barm|_{\partial M}$ in the splitting determined by $\rho$, so non-degeneracy
follows readily.

To prove the relations between the coefficients in the asymptotic form, we work in
more generality than needed here for later use. Putting
$\tau^{\rho}:=\frac{1}{\si^\rho}\in\Cal E(-1)$, we observe that $C=\tilde
C(\tau^{\rho})^2=\si^2g(\tau^\rho\rho\nu,\tau^\rho\rho\nu)$ and the right hand side
is a contraction between a smooth section of $S^2(T^b\barm)^*(2)$ and two sections of
$T^b\barm(-1)$. Hence we may differentiate the right hand side using the connection
$\nabla$, which admits a smooth extension to the boundary, and use that $\si^2g$ is
parallel for this connection. Hence for any vector field $\tilde\xi\in\frak
X(\barm)$, we get
$$
\tilde\xi\cdot C=2\si^2g(\tau^\rho\rho\nu,\nabla_{\tilde\xi}\tau^{\rho}\rho\nu). 
$$
Now over $M$, we can write
\begin{equation}\label{nu-deriv}
  \nabla_{\tilde\xi}\tau^\rho\rho\nu=\nabla^\rho_{\tilde\xi}\tau^\rho\rho\nu-\tau^\rho\tilde\xi=
  \tau^\rho(\rho\nabla^\rho_{\tilde\xi}\nu+d\rho(\tilde\xi)\nu-\tilde\xi),
\end{equation}
where we used that $\nabla^\rho \tau^\rho$=0, and observe that the sum of the last two terms inserts trivially in $d\rho$.
Plugging this into the above formula for $\tilde\xi\cdot C$ and using that $\nu$
inserts trivially into $\lambda$ and $h$, we see that
\begin{equation}\label{dC}
\tilde\xi\cdot
C= 2 Cd\rho(\nabla^\rho_{\tilde\xi}\nu)+\la(\rho\nabla^\rho_{\tilde\xi}\nu)-
\lambda(\tilde\xi).
\end{equation}
Observe that all terms in the right hand side indeed admit a smooth extension to the
boundary. Now given a vector field $\xi\in\frak X(\partial M)$, we can extend to
$\tilde\xi\in\frak X_{\partial}(\barm)$ such that $d\rho(\tilde\xi)\equiv 0$ and then
obtain $\xi\cdot\underline{C}$ as the boundary value of $\tilde\xi\cdot C$. But in
this case, $\nabla^\rho_{\tilde\xi}\nu$ is smooth up to the boundary, so the middle
term in the right hand side of \eqref{dC} is $\Cal O(\rho)$. Moreover, $\rho(\nu)$ is
constant and $\Rho(\tilde\xi,\nu)$ and (by Lemma \ref{lemma2.6}) $\Rho^\rho(\tilde\xi,\nu)$ admit a smooth
extension to the boundary, so \eqref{fund-form2} shows that the first term in the
right hand side is $\Cal O(\rho)$, too. Hence we get $\xi\cdot
\underline{C}=-\underline{\la}(\xi)$,
which shows that $\underline{\la}=-d\underline{C}$.

For the second identity, we start from a vector field $\eta\in\frak X(\partial M)$
and extend it smoothly to $\tilde\eta\in\frak X_{\partial}(\barm)$ such that
$d\rho(\tilde\eta)\equiv 0$. Together with the fact that $\nu$ inserts trivially into
$\la$ and $h$, this shows that
$\si^2g(\tau^\rho\rho\nu,\tau^\rho\tilde\eta)= \frac12\la(\tilde\eta)$. As
above, we conclude that for any $\tilde\xi\in\frak X(\barm)$ we get 
$$
\tilde\xi\cdot \la(\tilde\eta)= 2
\si^2g(\nabla_{\tilde\xi}\tau^\rho\rho\nu,\tau^\rho\tilde\eta)+ 
 2\si^2g(\tau^\rho\rho\nu,\nabla_{\tilde\xi}\tau^\rho\tilde\eta). 
$$
Now in the first term we can insert \eqref{nu-deriv} while in the second term
$d\rho(\tilde\eta)\equiv 0$ implies we can replace $\nabla$ by $\nabla^\rho$ and then take
$\tau^\rho$ out of the derivative. This shows that
\begin{equation}\label{nlambda}
  \tilde\xi\cdot \la(\tilde\eta)=d\rho(\nabla^{\rho}_{\tilde\xi}\nu)\la(\tilde\eta)+
   2 h(\rho\nabla^{\rho}_{\tilde\xi}\nu,\tilde\eta)- 2 h(\tilde\xi,\tilde\eta)+
   2 \frac{C}{\rho}d\rho(\nabla^\rho_{\tilde\xi}\tilde\eta)+
  \la(\nabla^\rho_{\tilde\xi}\tilde\eta).
\end{equation}
By \eqref{fund-form2}, since $\tilde\eta\in\frak X_{\partial}(\barm)$, and
$d\rho(\tilde\eta)\equiv 0$, $d\rho(\nabla^\rho_{\tilde\xi}\tilde\eta)=\Cal O(\rho)$
so all terms in the right hand side admit a smooth extension to the boundary.

As before, we need the formula here only in the case that $\tilde\xi\in\frak
X_{\partial}(\barm)$ is an extension of $\xi\in\frak X(\partial M)$ such that
$d\rho(\tilde\xi)\equiv 0$. In this case, the first term in the right hand side of
\eqref{nlambda} is $\Cal O(\rho)$ by \eqref{fund-form2}, and the second term is $\Cal
O(\rho)$ since $\nabla^{\rho}_{\tilde\xi}\nu$ admits a smooth extension to the
boundary. Finally, \eqref{fund-form2} together with $d\rho(\tilde\eta)\equiv 0$ and
the fact that $\Rho(\tilde\xi,\tilde\eta)=\Cal O(\rho)$ implies that
$d\rho(\nabla^\rho_{\tilde\xi}\tilde\eta)=\rho\Rho^{\rho}(\tilde\xi,\tilde\eta)+\Cal
O(\rho)$. Hence \eqref{nlambda} implies that
$$
h(\tilde\xi,\tilde\eta)=-\tfrac12 \tilde\xi\cdot
\la(\eta)+ \tfrac12 \la(\nabla^\rho_{\tilde\xi}\tilde\eta)+
C\Rho^\rho(\tilde\xi,\tilde\eta).  
$$
Passing to boundary values, we obtain the claimed equation for $h$ (taking into
account that, from the proof of Theorem \ref{thm2.8}, we know that the boundary value
of $\Rho^\rho$ is the Schouten tensor of the boundary connection).

For the last identity, we take $\eta,\zeta\in\frak X(\partial M)$ and extend to
$\tilde\eta,\tilde\zeta\in\frak X_{\partial}(\barm)$ which insert trivially into
$d\rho$. Then we immediately see that
$\si^2g(\tau^\rho\tilde\eta,\tau^\rho\tilde\ze)=h(\tilde\eta,\tilde\zeta)$, and as
before we conclude that for $\tilde\xi\in\frak X(\barm)$, we get
\begin{equation}\label{h-deriv}
\tilde\xi\cdot
h(\tilde\eta,\tilde\zeta)=\si^2g(\tau^\rho\nabla^\rho_{\tilde\xi}\tilde\eta,\tilde\zeta)
+\si^2g(\tau^\rho\tilde\eta,\tau^\rho\nabla^\rho_{\tilde\xi}\tilde\zeta), 
\end{equation}
where we replaced $\nabla$ by $\nabla^\rho$ as above. Expanding the first term in the
right hand side, we obtain
$$
\tfrac{1}{2\rho}d\rho(\nabla^{\rho}_{\tilde\xi}\tilde\eta)\lambda(\tilde\ze)+
h(\nabla^{\rho}_{\tilde\xi}\tilde\eta,\tilde\ze). 
$$
and by \eqref{fund-form2}, we can rewrite the first term as
$ \frac12 (\Rho^\rho(\tilde\xi,\tilde\eta)-\Rho(\tilde\xi,\tilde\eta))
\la(\tilde\ze)$. Bringing
the second term to the other side and adding the corresponding term from the other
summand in the right hand side of \eqref{h-deriv}, the left hand side becomes
$$
\tilde\xi\cdot
h(\tilde\eta,\tilde\ze)-h(\nabla^\rho_{\tilde\xi}\tilde\eta,\tilde\ze)-
h(\tilde\eta,\nabla^\rho_{\tilde\xi}\tilde\ze).
$$
Now if we specialize to the case that  $\tilde\xi$ is an extension of
$\xi\in\frak X(\partial M)$ which inserts trivially into $d\rho$, then the boundary
value of this of course is $(\nabla^\rho_\xi\underline{h})(\eta,\zeta)$. On the other
hand, this implies that $\Rho(\tilde\xi,\tilde\eta)=\Cal O(\rho)$, and the last
claimed identity follows. 
\end{proof}

\begin{remark}\label{rem3.3}
Observe that even fixing the defining function $\rho$, the coefficients $C$, $\la$
and $h$ in the asymptotic form \eqref{g-asymp2} are not uniquely determined beyond
the (restricted) boundary values $\underline{C}$, $\underline{\la}$ and
$\underline{h}$. For example, the expression remains unchanged if we replace $\la$ by
$\la+fd\rho+\rho\al$ for $f\in C^\infty(\barm,\Bbb R)$ and $\al\in\Om^1(\barm)$ and at
the same time replace $C$ by $C-\rho f$ and $h$ by $h-\al\odot d\rho$. One way to fix
the coefficients, which we have used implicitly in the proof already, is to choose a
vector field $\nu\in\frak X(\barm)$ such that $d\rho(\nu)\equiv 1$ and then require
that $\la(\nu)\equiv 0$ and $h(\nu,\ )\equiv 0$.

In particular, this can be applied to the form \eqref{g-asymp}, where
we see that the question of whether $C\mp 1$ is $\Cal O(\rho)$ or
$\Cal O(\rho^2)$ only makes sense once requires that the cross term
$\frac{1}{\rho^3}\la\odot d\rho$ vanishes identically. So this pins
down the coefficients here.
\end{remark}

 In particular, Theorem \ref{thm3.3} shows that, for projectively pre-compact metrics, the form \eqref{g-asymp}
can only be available on open subsets whose intersection with
$\partial M$ is either contained in $\partial M_+$ or in $\partial
M_-$ (corresponding to $C=1+\Cal O(\rho)$ or $-1+\Cal O(\rho)$). On
such subsets, one can easily obtain the form \eqref{g-asymp} (for an
appropriate defining function) from Theorem \ref{thm3.3} using the
ideas in Remark \ref{rem3.3}.

\begin{cor}\label{cor3.3}
For $\barm=M\cup\partial M$, let $g$ be a pseudo-Riemannian metric on $M$ which is
projectively pre-compact and asymptotically tangentially Ricci flat.

Then locally around each point of $\partial M_{\pm}$, there is a defining function
$\hat\rho$ with respect to which $g$ can be expressed as in \eqref{g-asymp} with
$C=\pm 1+\Cal O(\rho)$.
\end{cor}
\begin{proof}
Let us start from an expression \eqref{g-asymp2} for $g$ for a defining function
$\rho$. Then for the defining function $\hat\rho=e^f\rho$ we get $d\hat\rho=\hat\rho
df+e^fd\rho$ and inserting $\rho=e^{-f}\hat\rho$ and $d\rho=e^{-f}(d\hat\rho-\hat\rho
df)$ into \eqref{g-asymp2}, we see that we get an asymptotic form \eqref{g-asymp2}
for $g$ with respect to $\hat\rho$, which starts with
$\frac{e^{2f}C}{\hat\rho^4}d\hat\rho^2$. Now on the subset where $C$ is
non-vanishing, we can choose $f:=-\frac12\log(|C|)$ and so we get an expression of
the form
\begin{equation}\label{g-asymp3}
\frac{\pm 1}{\hat\rho^4}d\hat\rho^2+\tfrac{1}{\hat\rho^3}(\hat\la\odot
d\hat\rho)+\frac{1}{\hat\rho^2}\hat h. 
\end{equation}
Moreover, by Theorem \ref{thm3.3}, the boundary valued of $\hat\la$ vanishes in
tangential directions, so there are $\ph\in C^\infty(\barm,\Bbb R)$ and $\al\in\Om^1(\barm)$ such
that $\hat\la=\ph d\hat\rho+\hat\rho\al$. This means that \eqref{g-asymp3} equals
$$
\frac{\pm 1 +\rho\ph}{\hat\rho^4}d\hat\rho^2+\frac{1}{\hat\rho^2}(\hat
h+\al\odot d\hat\rho),
$$
so this has the claimed form.
\end{proof}

\subsection{Obstructions to projective compactness}\label{3.4}
We have already identified a complete obstruction to projective compactness in the
general setting of connections in Section \ref{2.3} via the boundary values of
$\rho\nabla^{\rho}_{\xi}\eta$ for $\xi,\eta\in\frak X(\barm)$. As we have seen there,
it is of the form $d\rho(\xi)|_{\partial M}d\rho(\eta)|_{\partial M}\psi$ for a
vector field $\psi\in\frak X(\partial M)$. Now for general asymptotic forms
\eqref{g-asymp2} we can use the computations from the proof of Theorem \ref{thm3.3}
to deduce an equivalent characterization of projective compactness for metrics.

\begin{thm}\label{thm3.4}
For $\barm=M\cup\partial M$, let $g$ be a projectively pre-compact pseudo-Riemannian
metric and assume that for some defining function $\rho$ and vector field
$\nu\in\frak X(\barm)$ such that $d\rho(\nu)\equiv 1$, we are given an expression
\eqref{g-asymp2} such that $\la(\nu)\equiv 0$ and $h(\nu,\ )\equiv 0$. (In particular
the asymptotic form constructed in the proof of Theorem \ref{thm3.3} has this
property.)

Then $g$ is projectively compact if and only if
\begin{equation}\label{C-eq}
  \nu\cdot C-2Cd\rho(\nabla^\rho_\nu\nu)=\Cal O(\rho)
\end{equation}
and for any $\tilde\ze\in\frak X(\barm)$ such that $d\rho(\tilde\ze)\equiv 0$, we get
\begin{equation}\label{lambda-eq}
-\nu\cdot\la(\tilde\ze)+\la(\nabla^\rho_\nu\tilde\ze)+
d\rho(\nabla^\rho_\nu\nu)\la(\tilde\ze)
+2\tfrac{C}{\rho}d\rho(\nabla^\rho_\nu\tilde\ze)=\Cal O(\rho). 
\end{equation}
\end{thm}
Observe that the first two terms in \eqref{lambda-eq} can be written as
$-(\nabla^\rho_\nu\la)(\tilde\ze)$ since $\nabla^\rho_\nu\la$ makes sense as acting
on $\frak X_{\partial}(\barm)$. Likewise, using \eqref{fund-form2}, we see that
$d\rho(\nabla^\rho_\nu\nu)=\rho\Rho^\rho(\nu,\nu)-\rho\Rho(\nu,\nu)$ and
$\frac1{\rho}d\rho(\nabla^\rho_\nu\tilde\ze)=\Rho^\rho(\nu,\tilde\ze)-\Rho(\nu,\tilde\zeta)$.
Correspondingly, we can interpret \eqref{C-eq} and \eqref{lambda-eq} as a system of
``asymptotic ODE'' on $C\in C^\infty(\barm,\Bbb R)$ and
$\lambda\in\Ga((T^b\barm)^*)$ of the form
$$
\nu\cdot C=2fC+\Cal O(\rho) \qquad \nabla^\rho_\nu\lambda=f\la+ 2 C\phi+\Cal O(\rho), 
$$
where $f:=\rho\Rho^\rho(\nu,\nu)-\rho\Rho(\nu,\nu)$ and
$\phi(\tilde\ze)=\Rho^\rho(\nu,\tilde\ze)-\Rho(\nu,\tilde\zeta)$.

\begin{proof}
The metric $g$ is projectively compact if and only if the boundary value
$\psi\in\Ga(T\partial M)$ of $d\rho(\rho\nabla^\rho_\nu\nu)$ vanishes identically. By
the non-degeneracy statement in Theorem \ref{thm3.3} this is equivalent to vanishing
of $\underline{\la}(\psi)$ and of $\underline{h}(\psi,\zeta)$ for any vector field
$\ze\in\frak X(\partial M)$. Now we can use formula \eqref{dC} from the proof of
Theorem \ref{thm3.3} (which is valid for any $\xi\in\frak X(\barm)$) in the case
$\xi=\nu$ to obtain
$$
 \la (\nabla^\rho_\nu\nu)=\nu\cdot C-2Cd\rho(\nabla^\rho_{\nu}\nu),
$$
which immediately shows that \eqref{C-eq} is equivalent to vanishing of
$\underline{\la}(\psi)$.

For \eqref{lambda-eq}, we extend $\ze$ to a vector field $\tilde\ze\in\frak X(\barm)$
such that $d\rho(\tilde\ze)\equiv 0$. Then we use formula \eqref{nlambda} from the
proof of Theorem \ref{thm3.3} (which is valid for any $\xi\in\frak X(\barm)$) in the
case $\xi=\nu$ to obtain
$$
 - h(\rho\nabla^{\rho}_\nu\nu,\tilde\ze)=-
\tfrac12 \nu\cdot\la(\tilde\ze)+ \tfrac12\la(\nabla^\rho_\nu\tilde\ze)+
 \tfrac12 d\rho(\nabla^\rho_\nu\nu)\la(\tilde\ze)
+\frac{C}{\rho}d\rho(\nabla^\rho_\nu\tilde\ze). 
$$ Passing to boundary values, \eqref{lambda-eq} is evidently equivalent to vanishing
of $\underline{h}(\psi,\zeta)$ for any $\ze\in\frak X(\partial M)$.
\end{proof}

So let us pass to the question of projective compactness of a metric of the form
\eqref{g-asymp}. As originally proved in \cite{Proj-comp} and reproved in Theorem
\ref{thm3.1} above, $C=\pm 1+\Cal O(\rho^2)$ is sufficient for projective
compactness. The asymptotic form \eqref{g-asymp} plays a prominent role in general
relativity, going back at least to \cite{Ashtekar-Romano}. There the $\Cal
O(\rho)$-term of $C\mp 1$ is referred to as the mass aspect and it plays a crucial
role in the applications to GR. So in that sense, projective compactness is
incompatible with mass, which has also been pointed out in \cite{Bo-He,Bo-He2,BCH}. From
this point of view, it would be desirable to prove that  for a metric of the form
\eqref{g-asymp} projective compactness forces $C=\pm 1+\Cal O(\rho^2)$. Our last
result shows that this is indeed true under slightly stronger assumptions on the
asymptotics of the Ricci curvature of $g$ (which should be satisfied in the cases
relevant for GR which mostly deal with Ricci-flat metrics).

\begin{thm}\label{thm3.4.2}
  For $\barm=M\cup\partial M$, let $g$ be a projectively pre-compact
  pseudo-Riemannian metric with an  asymptotic form \eqref{g-asymp}.
  Assume further that its Ricci curvature admits a
  smooth extension to the boundary and the boundary value of this curvature vanishes if
  one of the entries is tangent to the boundary.

  Then $g$ is projectively compact if and only if $C=\pm 1+\Cal O(\rho^2)$. 
\end{thm}
\begin{proof}
 It remains to show that projective compactness of $g$ forces $C=\pm 1+\Cal
 O(\rho^2)$, so we assume that $g$ is projectively compact.  We cannot directly use
 the computations from Theorem \ref{thm3.3} here, since in the case of
 \eqref{g-asymp}, we cannot say anything about $h(\nu,\ )$ for a vector field
 $\nu\in\frak X(\barm)$ such that $d\rho(\nu)\equiv 1$. It would be easy to redo
 these computations in the current setting, but we can also modify the asymptotic
 form along the lines of Remark \ref{rem3.3} and then use Theorem \ref{thm3.4}.
 
  Starting from \eqref{g-asymp} for a defining function $\rho$, we fix $\nu\in\frak
  X(\barm)$ such that $d\rho(\nu)\equiv 1$. Define $\phi:=h(\nu,\nu)\in
  C^{\infty}(\barm,\Bbb R)$ and for $\xi\in\frak X(\barm)$ consider
  $\al(\xi):=2(h(\nu,\xi)-d\rho(\xi)\phi )$. This clearly
  defines $\al\in\Om^1(\barm)$   such that $\al(\nu)\equiv 0$ and we define $\hat
  h\in\Ga(S^2T^*\barm)$ by $\hat h:=h-\al\odot d\rho - \ph d\rho^2$.
  Then $\hat h(\nu,\eta)=h(\eta,\nu)-\ph
    d\rho(\eta)+\tfrac12\al(\eta)=0$. Thus \eqref{g-asymp} can be rewritten as 
  $$
  g=\tfrac{C +\rho^2\ph}{\rho^4}d\rho^2+
  \tfrac{1}{\rho^3}(\rho\al)\odot d\rho+\hat
    h, 
    $$
    and by construction $\nu$ inserts trivially into $\rho\al$ and into $\hat
    h$. By Theorem \ref{thm3.4}, projective compactness implies that
    $$
\nu\cdot (C +\rho^2\ph)=2Cd\rho(\nabla^\rho_\nu\nu)+\Cal O(\rho), 
$$ and of course the left hand side equals $\nu\cdot C+\Cal O(\rho)$. As we have
noted after Theorem \ref{thm3.3},
$d\rho(\nabla^\rho_\nu\nu)=\rho\Rho^\rho(\nu,\nu)-\rho\Rho(\nu,\nu)$. By projective
compactness $\nabla^\rho$ admits a smooth extension to the boundary and hence so does
its Schouten-tensor so $\rho\Rho^\rho(\nu,\nu)=\Cal O(\rho)$. But by our stronger
curvature assumptions also $\Rho(\nu,\nu)$ admits a smooth extension to the boundary
and hence $\nu\cdot C=\Cal O(\rho)$. Since for $C=\pm 1+\rho f$, we get $\nu\cdot
C=f+\Cal O(\rho)$, this completes the proof.
\end{proof}

\subsection*{Data availability statement} Data availability is not applicable to this article 
as no new data were created or analyzed in this study.

\end{document}